\documentclass[11pt, reqno]{amsart}
\usepackage{latexsym}

\usepackage{amsmath}
\usepackage{color}

\definecolor{blu}{rgb}{0,0,0.1}

\makeatletter
\def\itemize{
  \ifnum\@itemdepth>3\@toodeep\else
    \advance\@itemdepth\@ne
    \edef\@itemitem{labelitem\romannumeral\the\@itemdepth}%
        \list{\csname\@itemitem\endcsname}%
      {\leftmargin=20pt\def\makelabel##1{\hss\llap{##1}}}
        \fi}

\renewenvironment{enumerate}{%
  \ifnum \@enumdepth >3 \@toodeep\else
      \advance\@enumdepth \@ne
      \edef\@enumctr{enum\romannumeral\the\@enumdepth}\list
      {\csname label\@enumctr\endcsname}{\usecounter
        {\@enumctr}\leftmargin=20pt\def\makelabel##1{\hss\llap{\upshape##1}}}\fi
}{%
  \endlist
}

\makeatletter
\def\@settitle{\begin{center}%
  \baselineskip14\p@\relax
    \bfseries\@title
  \end{center}%
}
\def\@setauthors{%
  \begingroup
  \def\thanks{\protect\thanks@warning}%
  \trivlist
  \centering\footnotesize \@topsep30\p@\relax
  \advance\@topsep by -\baselineskip
  \item\relax
  \author@andify\authors
  \def\\{\protect\linebreak}%
  \authors%
  \ifx\@empty\contribs
  \else
    ,\penalty-3 \space \@setcontribs
    \@closetoccontribs
  \fi
  \endtrivlist
  \endgroup
}
\def\@setthanks{\def\thanks##1{\par##1}\thankses}
\makeatother

   \makeatletter
   \def\LaTeX{\leavevmode L\raise.42ex
       \hbox{\kern-.3em\size{\sf@size}{0pt}\selectfont A}\kern-.15em\TeX}
   \makeatother
   
   \newcommand{\BibTeX}{{\rm B\kern-.05em{\sc
             i\kern-.025emb}\kern-.08em\TeX}}
\def\bbm[#1]{\mbox{\boldmath $#1$}}



   \newcommand{\e }{\varepsilon }
   \newcommand{\intr }{\int_{\R^2}}
   \newcommand{\into }{\int_{\Omega}}

   \newcommand{\R}{{\mathbb{R}}}

\newcommand{\cal}{\mathcal }  
   
   \newcommand{\N}{\mathbb{N}}
   \newcommand{\beq}{\begin{equation}}
   \newcommand{\eeq}{\end{equation}}

  \renewcommand{\theequation}{\arabic{section}.\arabic{equation}}
   \renewcommand{\theequation}{\thesection.\arabic{equation}}
   \newtheorem{theorem}{Theorem}[section]
   
   \newtheorem{proposition}[theorem]{Proposition}
   \newtheorem{lemma}[theorem]{Lemma}
   \newtheorem{corollary}[theorem]{Corollary}
   \newtheorem{remark}[theorem]{Remark}
   \newcommand{\bremark}{\begin{remark} \em}
   \newcommand{\eremark}{\end{remark} }

\DeclareMathOperator{\di}{d}
\DeclareMathOperator{\dist}{dist}
\def\bbm[#1]{\mbox{\boldmath $#1$}}
\def\bbm[#1]{\mbox{\boldmath $#1$}}

\begin{document}

\title[Blow-up solutions]
{\LARGE Multiple blow-up solutions for the\\ Liouville equation with singular data}

\author[Teresa D'Aprile]{\Large Teresa D'Aprile}
\address{Teresa D'Aprile, Dipartimento di Matematica, Universit\`a di Roma ``Tor
Vergata", via della Ricerca Scientifica 1, 00133 Roma, Italy.}
\email{daprile@mat.uniroma2.it }

\thanks{The author has been supported by  the Italian PRIN Research Project 2009
\textit{Metodi variazionali e topologici nello studio dei fenomeni non lineari}}

\begin{abstract} We study the existence of solutions with multiple concentration 
 to the following boundary value problem $$-\Delta u=\e^2 e^u-4\pi \sum_{p\in Z}\alpha_p \delta_{p}\;\hbox{ in } \Omega,\quad	u=0 \;\hbox{ on }\partial \Omega,$$  where $\Omega$ is a smooth  and bounded domain in $\R^2$,  $\alpha_{p}$'s are positive numbers, $Z\subset \Omega$ is a finite set, $\delta_p$ defines the Dirac mass at $p$,  and $\e>0$ is a small parameter. In particular we extend the result of Del-Pino-Kowalczyk-Musso (\cite{delkomu}) to the case of several singular sources. More precisely we prove that, under suitable restrictions on the weights $\alpha_p$,
a solution exists with a  number of  blow-up points $\xi_j\in \Omega\setminus Z$  up to $\sum_{p\in Z}\max\{n\in\N\,|\, n<1+\alpha_p\}$.

 \bigskip



\end{abstract}
\maketitle
\section{Introduction}
Let $\Omega$ be a bounded domain in $\R^2$ with a smooth boundary. In this paper we are concerned with the existence and the asymptotic analysis when the parameter $\e$ tends to $0$ of solutions in the distributional sense for the following Liouville-type problem involving singular sources:
\beq\label{sinh}\left\{\begin{aligned} &-\Delta u=\e^2 e^u-4\pi \sum_{p\in Z}\alpha_p \delta_{p}& \hbox{ in } &\Omega\\ &u=0 &\hbox{ on } &\partial \Omega\end{aligned}\right..\eeq Here $\delta_p$ denotes Dirac mass supported at $p$, $\alpha_{p}$'s are positive numbers and  $Z\subset \Omega$ is a finite set of  singular sources. 
Problem \eqref{sinh} is motivated by its links with the modeling of physical phenomena. In particular,  \eqref{sinh} arises in the study of vortices   in a planar model of Euler flows (see \cite{bapi}, \cite{delespomu}). In vortex theory the interest in constructing \textit{blowing-up} solutions is related to relevant physical properties, in particular  the presence of vortices with a strongly localised electromagnetic field.

The regular problem, when $Z=\emptyset$, has been widely considered in literature. The asymptotic behaviour of blowing up family 
of solutions 
     can be referred to the papers \cite{bapa}, \cite{breme}, \cite{lisha}, \cite{mawe},  \cite{nasu}, \cite{su}. 
   The analysis in these works yields that if $u_\e$ is    an unbounded family of solutions  of \eqref{sinh}  with $Z=\emptyset$ for which   $\e^2\into e^{u_\e} $ is uniformly bounded, then there is an integer $N\geq 1$ such that 
   $$\e^2\into e^{u_\e}dx\to 8\pi N\hbox{ as }\e\to 0.$$ Moreover there are points $\xi_1^\e,\ldots, \xi_N^\e\in\Omega$ which remain uniformly distant from the boundary $\partial\Omega$ and from one another such that  \beq\label{issue}\e^2e^{u_\e}-\sum_{j=1}^N \delta_{\xi_j^\e}\to 0\eeq in the measure sense. 
   The location of the     blowing-up points is well understood.
    Indeed, in \cite{nasu} and \cite{su} it is established that the $N$-tuple $(\xi_1^\e,\ldots, \xi_N^\e)$ converges, up to a subsequence, to a critical point of the functional 
  \beq\label{natur}\frac12\sum_{j=1}^NH(\xi_j,\xi_j)+\frac12\sum_{j,k=1\atop j\neq k}^NG(\xi_j,\xi_k).\eeq  
 Here 
   $G(x,y)$ is the Green's function of $-\Delta$ over $\Omega$ under Dirichlet
boundary conditions, namely   $G$ satisfies
\begin{equation*}
\left\{\begin{aligned}
&-\Delta_xG(x,y)=\delta_y(x) &\hbox{ }&x\in\Omega\\
&G(x,y)=0 &\hbox{ }&x\in\partial \Omega
\end{aligned}\right.,
\end{equation*}
and  $H(x,y)$ denotes its regular part:
$$H(x,y)=G(x,y)-\frac{1}{2\pi}\log\frac{1}{|x-y|}.$$
   The reciprocal issue, namely the existence of positive solutions with the property  \eqref{issue}, has been addressed in \cite{we} in the case where $\Omega$ is simply connected and where there is a single point of concentration (i.e. $N=1$). Baraket and Pacard (\cite{bapa}) established that for any nondegenerate critical point of \eqref{natur} a family of solutions $u_\e$ concentrating at the corresponding $N$-tuple as $\e\to 0$ does exist. See also \cite{chenlin} for an alternative proof in a compact two-dimensional Riemannian manifold.
   More generally, in \cite{delkomu} and \cite{espogropi} it is proved that any \textit{topologically nontrivial} critical point of \eqref{natur} generates a family of solutions with multiple blow-up points.  In particular, such a family of solutions is found in some special cases: for any $N\geq 1$, provided that  $\Omega$ is not simply connected  (\cite{delkomu}), and  for $N\in\{1,\ldots, h\}$ if $\Omega$ is a $h$-dumbell with thin handles (\cite{espogropi}). 
   
   We mention that an analogous blow-up analysis can be applied to sign changing solutions of the $\sinh$-Poisson equation $-\Delta u=\e^2 (e^u-e^{-u})$ allowing also negative concentration phenomena; in this case    the coupling  of bump pairs  is not described  anymore by the Green's function $G(\xi_j,\xi_k)$ in \eqref{natur},  but 
  the coupling    
  depends  on the  respective sign of the bumps and actually becomes $\tau_j\tau_kG(\xi_j,\xi_k)$ with $\tau_j=\pm 1$.  Multiple blow-up families of nodal solutions always exist for the $\sinh$-Poisson equation in a general  bounded and smooth domain $\Omega$. This result has become known first for exactly one positive and one negative blow-up points  in \cite{bapi} and, more recently, for  three and four alternate-sign blow-up points in \cite{bapiwe}. 
   Similar constructions have been carried out    for 
   the Henon-type equation $-\Delta u=\e^2 |x|^{2\alpha}(e^u-e^{-u})$ (\cite{mio}).

In all the above papers, only the regular case $Z=\emptyset$ is considered. Motivated by some singular limiting model equations  arising in the study of Chern-Simons vortex theory (see \cite{brebehe}, \cite{yang} and the references therein), in this paper we are interested in analysing \eqref{sinh} when $Z\neq \emptyset$. 

Concerning construction of solutions for the problem \eqref{sinh} in the singular  case $Z\neq\emptyset$, only  few results are available in literature.
 An extension to the problem \eqref{sinh} of the blow-up analysis carried out for $Z=\emptyset$
 is contained in \cite{bapa}. 
 It permits to perform an asymptotic analysis similar to the regular case. More precisely, let us  define the new functional 
\beq\label{psi0}\Psi(\bbm[\xi])= \frac12\sum_{j=1}^NH(\xi_j,\xi_j)- \sum_{p\in Z}\frac{\alpha_p}{2} \sum_{j=1}^N G(\xi_j,p)+\frac12\sum_{j,k=1\atop j\neq k}^NG(\xi_j,\xi_k)\eeq which is well defined in the set
\beq\label{emme}{\cal M}:=\Big\{\bbm[\xi]=(\xi_1,\ldots, \xi_N)\in (\Omega\setminus Z)^N\,\Big|\,  \xi_j\neq \xi_k\,\hbox{ if } j\neq k\Big\}.\eeq
Then the functional $\Psi$ replaces \eqref{natur} in locating the points in ${\cal M}$ where the concentration occurs for problem \eqref{sinh}.

Assuming that $Z$ is a singleton, i.e. $Z=\{p\}$, then   solutions which concentrate in the measure sense at  $N$ distinct points away from $p$ have been built in \cite{delkomu} provided that $N<1+\alpha_p$. 
 Moreover  in \cite{delespomu}, for any fixe $\alpha\in \N$,   it is proved the existence of  a solution to \eqref{sinh} with  $Z=\{p\}$ and $\alpha_p=\alpha$ for  a suitable $p\in \Omega$   with $\alpha+1$ blowing up points  at the vertices of a sufficiently tiny regular polygon  centered in $p$; moreover $p$ lies uniformly away from the boundary $\partial \Omega$ independently of $\alpha$  and its location is determined by the geometry of the domain.

In this paper we address the question of  which kind of 
concentration phenomena 
can be observed in $\Omega\setminus Z$  in the case of multiple sources.  
 One  expects that, if   several sources exist, then the more involved topology should generate a larger number of blow-up solutions than  the singleton case considered in  \cite{delkomu}.   As far as we know, up to now this kind of multiplicity result has remained unknown.

The aim of this paper is to establish the existence of  a family of solutions to \eqref{sinh} consisting of $\sum_{p\in Z}N_p$ blow up points in $\Omega\setminus Z$ as long as $N_p< 1+\alpha_p $ for all $p\in Z$, provided that the weights $\alpha_p$ avoid the integers $1,2,\ldots, N-1$. 
More precisely we will prove the following theorem.
\begin{theorem}\label{th1} Assume that $\Omega\subset \R^2$ is a bounded domain with a ${\cal C}^2$ boundary, $Z\subset \Omega$ is a finite set and $\alpha_p>0$ for all $p\in Z$. 
 Let $N\geq 2 $ be an integer satisfying the following:
\beq\label{akka2}\alpha_p\neq 1,\ldots, N-1\;\;\forall p\in Z,\eeq
\beq\label{akka1}\hbox{there exist integers } N_p\hbox{ with }0\leq N_p< 1+\alpha_p \hbox{ s.t. }N=\sum_{p\in Z}N_p.\eeq
 Then for $\e>0$ sufficiently small the problem \eqref{sinh} has a family of solutions $u_\e$      with the following property: 
there exist distinct points $\xi_1^\ast, \ldots ,\,\xi_N^\ast\in\Omega\setminus Z$   such that 
\beq\label{issue1}\e^2e^{u_\e}\to 8\pi \sum_{j=1}^N\delta_{\xi_j^\ast}\;\hbox{ as }\e\to 0\eeq in the measure sense. Furthermore $u_\e$ verifies 
$$u_\e(x)=-4\pi\sum_{p\in Z} \alpha_p G(x,p)+8\pi \sum_{j=1}^N G(x, \xi_j^\ast) +o(1)\;\hbox{ as }\e\to 0$$
 away from the concentration points $\xi_1^\ast, \ldots ,\,\xi_N^\ast$.
\end{theorem}

Let us observe that \eqref{akka2} implies $\alpha_p\neq 1$; moreover \eqref{akka2} is always verified if $\alpha_p\not \in \N$ for all $p\in Z$. 
 In particular, the assumptions \eqref{akka2}-\eqref{akka1} are satisfied  for $2\leq N<1+\alpha_p$ in the singleton case $Z=\{p\}$, so  
  Theorem \ref{th1}   covers the result known in \cite{delkomu}.
  
  We point out that in the above result  
   concentration occurs 
    at points different from the location of the sources. The problem of finding solutions with additional concentration around the sources is of different nature. In case they exist, each blow-up point  provides a contribution of $8\pi(1+\alpha_p)\delta_p$ in the limit of \eqref{issue1}, see \cite{bachelinta},  \cite{bata}, \cite{ta1}, \cite{ta2}.  This problem  requires a different asymptotic analysis which allows in \cite{espo} the construction of solutions concentrating at any prescribed subset $\tilde{Z}\subset Z$ under the additional assumption $\alpha_p\not \in \N$ for all $p\in Z$. 
   
  The proof of Theorem \ref{th1} relies on the finite dimensional variational reduction developed in \cite{delkomu}, which 
reduces the problem of constructing blowing up families of solutions for \eqref{sinh} to the problem of finding critical points of a functional depending on the blow-up points $\xi_j$. More specifically, the reduced functional is ${\cal C}^1$-close to the functional $\Psi$
  given by \eqref{psi0}. 
In this paper  we will use a variational argument to characterize a \textit{topologically nontrivial} cri\-ti\-cal point of $\Psi$ which will \textit{survive} small ${\cal C}^1$-perturbations. The  notion of \textit{topologically nontrivial} critical point we use here was introduced in \cite{delfe} in the analysis of concentration phenomena in nonlinear Schr\"odinger equations.  Similar notions have been  widely used to detect multiple-bubbling solutions  in other related singularly perturbed problems.

In the expression of $\Psi$ we recognize the bump interaction terms $G(\xi_j,\xi_k)$ which raise the energy to $+\infty$  if any two bumps get very close, while the Robin function $H(\xi_j,\xi_j)$ describes the interaction of an individual bump with the boundary and pulls the energy to $-\infty$ if the point approaches the boundary. 
The presence of the singular set $Z$ in $\Omega$ which repel the points produces the required nontrivial topology which intuitively yields an equilibrium generated  a max-min situation for points as distant as possible from one to another, as well as from the boundary and from  the singularities. 
Indeed, we will prove  that a $N$-bump solution exists for any $N\geq 2$ satisfying \eqref{akka2}-\eqref{akka1}, with bumps associated to such an equilibrium.

In order to understand the mechanism which drives energy equilibrium, we briefly outline the variational argument, which consists of two parts.  First, we construct   a special class $K$ of initial geometric configurations of the blow-up points: more precisely,  $K$ consists  of the configurations in ${\cal M}$ obtained by placing $N_p$ points on a small annulus centered in $p$  for any $p\in Z$; such a geometrical construction  requires that $N$ can be split as in \eqref{akka1}. The class $K$ has the crucial property that the terms $H(\xi_j,\xi_j)$ and $G(\xi_j,p)$ are uniformly bounded, then the restriction $\Psi\big|_{K}$ is bounded from below and satisfies  $\Psi\big|_{K}\to +\infty$ as $\bbm[\xi]\to \partial {\cal M}$. Therefore $\Psi$ attains a global minimum in $K$. Secondly, we apply a localized max-min technique on suitable deformation in ${\cal M}$ of the class $K$. 

Once we have provided the functional $\Psi$ with a suitable local max-min structure, we   need some compactness to conclude that the max-min value actually gives the existence of a critical point. This is the most technical part of the proof and requires the additional assumption \eqref{akka2}.
Roughly speaking, 
the major problem in proving the validity of compactness
is to exclude possible 
collisions, i.e. 
configurations $\bbm[\xi] = (\xi_1, \ldots, \xi_N)$ in $\cal M$ with at least two components
converging towards the same point of $\Omega$. 
More precisely, compactness may fail when a collision arises near  some singularity  $p\in Z$ for bounded values of  $\Psi$ and $\nabla \Psi$. Such  \textit{crucial} collisions do occur when $\alpha_p=n\in \{1,\ldots, N-1\}$ for some $p\in Z$ and $n+1$ of the $N$ points are arranged on a tiny polygonal  configuration centered at $p$. This explains the assumption \eqref{akka2}. The result provided by Theorem \ref{th1} is optimal in the sense that if some weight  $\alpha_p$ coincides with one of the integers $1,\ldots, N-1$, then  the method breaks down due to lack of compactness.

It is interesting to notice that the function $\Psi$ appears frequently in the study of problems arising from physics or geometry. For instance, in the context of Ginzburg-Landau vortices $\Psi$   is also called the \textit{renormalized energy}. It is a generalization of the classical Kirchhoff-Routh path function for $\Omega=\R^2$ and  describes the dynamics of point vortices in an incompressible planar fluid.  The  \textit{renormalized energy} is a Hamiltonian and  the motion of the vortices is governed by the corresponding Hamiltonian equation. Therefore   
the problem of finding critical points of $\Psi$ corresponds to  finding  all relative equilibrium configurations of $N$-point vortices in $\Omega$. We refer to the book \cite{brebehe} for
a detailed work in this direction.

Finally we point out that  $\Psi$ occurs as limit equation in a variety of problems involving concentration phenomena with multi point concentration. We mention \cite{weiyezhou} for the context of Liouville equation with anisotropic coefficients, and  \cite{espowei} for the $\sinh$-Gordon equation. See also   
\cite{espomupi1}, \cite{espomupi2}, \cite{espopiwei} 
 where the role of the function $\Psi$  appears in the construction of solutions with multiple concentration for a planar elliptic equation involving nonlinearities with large exponent. 
 
The paper is organized as follows. In Section 2 we sketch the finite-dimensional reduction method developed in \cite{delkomu}. Section 3 is devoted to solving the reduced problem by carrying out the max-min procedure. Finally in the Appendix A we collect some properties of the Green's function which are usually referred to throughout the paper.

\begin{remark} In the sequel of the paper, by  $c$, $C$ we will denote generic  positive constants. The value of $c,\, C$ is allowed to vary from line to line and also in the same formula.\end{remark}

\section{The reduced functional}
The proof of Theorem \ref{th1}  is based on the \textit{finite dimensional reduction}
procedure which has been used for a wide class of singularly perturbed problems. 

Let us observe that, setting  $v$ the regular part of $u$, namely $$v= u+4\pi \sum_{p\in Z}\alpha_p G(x, p),$$  problem \eqref{sinh} is then equivalent to solving the following (regular) boundary value problem
\beq\label{proreg}\left\{\begin{aligned} & -\Delta v=\e^2a(x)e^v&\hbox{ in }&\Omega\\ &v=0&\hbox{ on }&\partial \Omega\end{aligned}\right.,\eeq
where $a(x)$ is the function \beq\label{aaa}a(x)=e^{-4\pi \sum_{p\in Z} \alpha_pH(x,p)}\Pi_{p\in Z} |x-p|^{2\alpha_p}.\eeq
Here $G$ and $H$ are Green's function and its regular part as defined in the introduction. 

 In \cite{delkomu} the reduction process has been applied to  the problem \eqref{proreg}  for a more general function $a$.
We sketch
the procedure here and refer to  \cite{delkomu}  for details.

 The first object is to construct an approximate solution for  \eqref{proreg}. To this aim, 
let us introduce the  radially symmetric solutions of the limiting problem 
$$\left\{\begin{aligned}&-\Delta \omega=e^\omega\;\hbox{ in } \R^2\\ &\intr e^\omega dx<+\infty\end{aligned}\right.,$$
which are given by the one-parameter  family of functions$$
\omega_{\mu}(x)
 =\log \frac{8\mu^2}{(\mu^2+|x|^2)^2},\quad \mu>0
$$ (see \cite{chenlin}).  We point out that for $\xi\in \Omega\setminus Z$  the function 
$$U_{\e,\mu, \xi}(x)=\log\frac{8\mu^2}{(\mu^2\e^2+|x-\xi|^2)^2 a(\xi)}=\omega_{\mu}\Big(\frac{x-\xi}{\e}\Big)+4\log\frac{1}{\e}-\log a(\xi)$$ satisfies$$-\Delta U=a(\xi)\e^2 e^U\;\hbox{ in }\R^2.$$
Let $N$ be a positive integer and  define the configuration space
$$\cal{O} :=
\left\{ \bbm[\xi]
      =( \xi_1,\ldots,\xi_N)\in\Omega^N\,\big|\,
 \di(\xi_j)>\delta\,\, \&\,\,|\xi_j-p|>\delta \;\forall j,\,\forall p\in Z,\; \;|\xi_j-\xi_k|\geq \delta\hbox{ if }j\neq k
\right\}
$$
where $\delta>0$ is a sufficiently small number and  ${\rm{d}}(x):={\dist}(x,\partial\Omega)$. It is natural to look for a solution $v$ of \eqref{proreg} of the form $v\approx \sum_{j=1}^NU_{\e,\mu_j,\xi_j}$               for a certain set of points $\bbm[\xi]=(\xi_1,\ldots, \xi_N)\in {\cal O}$ and suitable scalars $\mu_1,\ldots, \mu_N>0$. 
So, we would like to take $\sum_{j=1}^NU_{\e,\mu_j,\xi_j}$ as our first approximate solution of the problem \eqref{proreg}. To obtain a better first approximation, we need to modify it in order to satisfy the zero boundary condition. 
Precisely, we
consider the projections ${\cal P}_\Omega U_{\e,\mu, \xi}$ onto the space $H^1_0(\Omega)$ of
$U_{\e,\mu, \xi}$, where the projection  ${\cal P}_\Omega:H^1(\R^N)\to H^1_0(\Omega)$ is
defined as the unique solution of the problem
$$
\left\{
\begin{aligned}
&\Delta {\cal P}_\Omega v=\Delta v &\hbox{ in } &\Omega\\
& {\cal P}_\Omega v=0&\hbox{ on }&\partial \Omega
\end{aligned}\right..$$
Finally, in order to make the approximation  $\sum_{j=1}^N{\cal P}_\Omega U_{\e,\mu_j, \xi_j}$ more accurate near $\xi_j$, we can improve the construction by further adjusting the numbers $\mu_j$. Actually, the parameter $\mu_j$ will be a priori prescribed in terms of the point $\xi_j$ by the formula:
\beq\label{formula}\log 8\mu_j^2=\log a(\xi_j) +8\pi H(\xi_j,\xi_j)+8\pi\sum_{k=1\atop k\neq j}^NG(\xi_j,\xi_k).\eeq Then we consider the first approximation 
$$\sum_{j=1}^N{\cal P}_\Omega U_{\e,\xi_j}$$ where $U_{\e,\xi_j}=U_{\e,\mu_i,\xi_j}$ with the numbers $\mu_j=\mu_j(\xi_j)$ defined by \eqref{formula}. Let us analyse the asymptotic behavior of ${\cal P}_\Omega U_{\e,\xi_j}$ with these choices of $\mu_j$ as $\e\to 0^+$. For any  $x\in\partial \Omega$ we compute     \beq\label{sater}\begin{aligned}{\cal P}_\Omega U_{\e, \xi_j}-U_{\e, \xi_j}-8\pi H(x, \xi_j)&=-\log\frac{8\mu_j^2}{(\mu_j^2\e^2+|x-\xi_j|^2)^2 a(\xi_j)}+4\log\frac{1}{|x-\xi_j|}\\ &=-\log\frac{8\mu_j^2}{a(\xi_j)}+O(\e^2)\end{aligned}\eeq uniformly on  $\partial \Omega$.
So, the functions ${\cal P}_\Omega U_{\e, \xi_j}-U_{\e,\xi_j}$ and $8\pi H(x, \xi_j)$ are harmonic in $\Omega$ and satisfy \eqref{sater} on the boundary. Then the maximum principle applies and gives \beq\label{estione}{\cal P}_\Omega U_{\e, \xi_j}=U_{\e, \xi_j}+8\pi H(x, \xi_j)-\log\frac{8\mu_j^2}{a(\xi_j)}+O(\e^2)\,\hbox{ uniformly in }\Omega.\eeq
Moreover for any $k$ a direct computation yields
$$U_{\e, \xi_k}=
4\log\frac{1}{|x-\xi_k|} +\log\frac{8\mu_k^2}{a(\xi_k)}+O(\e^2) \hbox{ uniformly for } |x-\xi_k|\geq \frac{\delta}{2} $$ by which 
\beq\label{estitwo}{\cal P}_\Omega U_{\e, \xi_k}=
8\pi G(x, \xi_k) +O(\e^2) \hbox{ uniformly for } |x-\xi_k|\geq \frac{\delta}{2} .\eeq
Therefore, fixed $j\in\{1,\ldots, N\}$, combining \eqref{estione} and \eqref{estitwo} we have 
$$\sum_{k=1}^N{\cal P}_\Omega U_{\e,\xi_k}=U_{\e,\xi_j}+8\pi H(x, \xi_j)-\log\frac{8\mu_j^2}{a(\xi_j)}+8\pi\sum_{k=1\atop k\neq j}^N G(x,\xi_k)+O(\e^2)$$
uniformly for $|x-\xi_j|\leq \frac{\delta}{2}.$
The above estimate implies that $\sum_{k=1}^N{\cal P}_\Omega U_{\e,\xi_k}\sim U_{\e,\xi_j}$ for $x$ close to $\xi_j$  thanks to the choice of $\mu_j$ in \eqref{formula}. This justifies the \textit{good} choice of the numbers $\mu_j$ in \eqref{formula}. 
Now  we look for a solution
to \eqref{proreg} in a small neighbourhood of the first approximation, i.e. a solution of the form
$$
v:=\sum_{j=1}^N {\cal P}_\Omega  U_{\e,\xi_j}+\phi,
$$
where the rest term $\phi$ is small. To carry out the construction of a solution of this type,
we first introduce an intermediate problem as follows.
Let us set
$$Z_{j,k}(x)=\frac{\partial  }{\partial \xi_j^k}\omega_{\mu_j }\Big(\frac{x-\xi_j}{\e}\Big), \quad j=1,\ldots, N,\;k=1,2;$$
here we denote by $\xi_j^k$ the $k$-th component of $\xi_j$. Additionally, let us consider a sufficiently large radius $R_0>0$ and a nonnegative function $\chi$ such that $\chi=1$ if $\rho<R_0$ and $\chi(\rho)=0$ if $\rho>R_0+1$ and set $\chi_\e(\rho)=\chi(\frac{\rho}{\e}).$ 
Then it is convenient to solve
as a first step the problem for $\phi$ as a function of $\e$ and $\bbm[\xi]$.
This turns out to be solvable for any choice of points $\xi_j$,
provided that $\e$ is sufficiently small. The following result was established in \cite{delkomu}.

\begin{lemma}\label{reg}
There exists $\e_0>0$ and a constant $C>0$ such that for each $\e\in (0,\e_0)$ and each
$\bbm[\xi]\in {\cal O}$ there exists a unique
$\phi_{\e,\hbox{\scriptsize$\bbm[\xi]$} }
 \in H^1_0(\Omega)$ satisfying
\begin{equation}\label{sati1}
\Delta\bigg(\sum_{j=1}^N {\cal P}_\Omega  U_{\e,\xi_j}+\phi\bigg)
 +\e^2a(x)\exp \bigg(\sum_{j=1}^N {\cal P}_\Omega  U_{\e,\xi_j}+\phi\bigg)
 =\!\!\sum_{j=1,\ldots, N\atop k=1,2}c_{j,k} \chi_\e(|x-\xi_j|)Z_{j,k}
\end{equation}for some $c_{j,k}\in \R$ and 
\begin{equation}\label{sati11}\into \chi_\e(|x-\xi_j|)Z_{j,k} \phi\, dx=0\qquad \forall j=1,\ldots, N,\;k=1,2,\eeq
\begin{equation}\label{sati2}
\|\phi\|_\infty< C\e|\log \e|.
\end{equation}
 Moreover the map $\bbm[\xi]\in {\cal O}\mapsto\phi_{\e,\hbox{\scriptsize$\bbm[\xi]$} }\in H^1_0(\Omega)$ is ${\mathcal C}^1$.
\end{lemma}

After problem \eqref{sati1}--\eqref{sati2} has been solved, then we find a solution to the problem \eqref{proreg} if $\bbm[\xi]\in {\cal O}$ is such that $c_{j,k}=c_{j,k}(\bbm[\xi])=0$ for all $j,\,k$. This problem is actually variational. Indeed, let us consider  the following energy functional associated with 
\eqref{proreg}:
\begin{equation}\label{func1}
I_\e(v)=\frac12\into |\nabla v|^2 dx-\e^2\into a(x)e^v\, dx ,\quad v\in H^1_0(\Omega).
\end{equation}
Solutions of \eqref{proreg} correspond to critical points of $I_\e$. Now we introduce the new
functional
\begin{equation}\label{jeps}
J_\e:{\cal O}\to \R, \quad J_\e(\bbm[ \xi])
 =I_\e\bigg(\sum_{j=1}^N {\cal P}_\Omega  U_{\e,\xi_j}
       +\phi_{\e,\hbox{\scriptsize$\bbm[\xi]$} }\bigg)
\end{equation}
where $\phi_{\e,\hbox{\scriptsize$\bbm[\xi]$}}$ has been constructed in
Lemma \ref{reg}. The next lemma has been proved in \cite{delkomu} and reduces the 
problem \eqref{proreg} to the one of finding critical points of the functional $J_\e$.

\begin{lemma}\label{relation}
 $\bbm[\xi]  \in {\cal O}$ is a critical point of $J_\e $ if and
only if the corresponding function
$v_\e=\sum_{j=1}^N {\cal P}_\Omega  U_{\e,\xi_j}
       +\phi_{\e,\hbox{\scriptsize$\bbm[\xi]$} }$
is a solution of (\ref{proreg}). Consequently, the function $$u_\e=-4\pi \sum_{p\in Z}\alpha_p G(\cdot, p)+\sum_{j=1}^N {\cal P}_\Omega  U_{\e,\xi_j}
       +\phi_{\e,\hbox{\scriptsize$\bbm[\xi]$} }$$ is a solution of the original problem \eqref{sinh}.
\end{lemma}

Finally we describe an expansion for $J_\e$ which can be obtained as in
\cite{delkomu}.

\begin{proposition}\label{exp1}
The following asymptotic
expansion holds:
$$
J_\e (\bbm[\xi] )
= -16N\pi +8N\pi\log 8-16 N\pi \log \e +4\pi \Psi(\bbm[\xi])+O(\e)
$$
$\cal C^1$-uniformly with respect to  $\bbm[\xi]\in{\cal O}$. 
Here the function $\Psi$ is defined by 
\beq\label{pssi}\Psi( \bbm[\xi])
=\frac12\sum_{j=1}^NH(\xi_j,\xi_j)+\frac{1}{8\pi}\sum_{j=1}^N \log a(\xi_j)+\frac12\sum_{j,k=1\atop j\neq k}^NG(\xi_j,\xi_k)
.\eeq 
\end{proposition}

Thus in order to construct a solution of problem \eqref{sinh} such as the one predicted in
Theorem \ref{th1} it remains to find a critical point of $J_\e$. This will be accomplished in
the next section.

\section{A max-min argument: proof of Theorem \ref{th1}}
In this section we will employ the reduction approach to construct the solutions stated in
Theorem \ref{th1}. 
The results obtained in the previous section imply that our problem reduces to investigate the existence of  critical points of a functional which is a small ${\cal C}^1$-perturbation of  the function $\Psi$ given by \eqref{pssi}, with $a$ defined by \eqref{aaa}. So $\Psi$ actually becomes 
$$\Psi( \bbm[\xi])
=\frac12\sum_{j=1}^NH(\xi_j,\xi_j)-\sum_{p\in Z}\frac{\alpha_p}{2}\sum_{j=1}^N G(\xi_j,p)+\frac12\sum_{j,k=1\atop j\neq k}^NG(\xi_j,\xi_k).
$$ 
We recall that $\Psi$  is well defined in the set  ${\cal M}$ defined in \eqref{emme}. 
In this section we apply a max-min argument to characterize a
topologically nontrivial  critical value of this function in the set ${\cal M}$. More precisely
we will construct sets $\mathcal D$, $K$, $K_0\subset
{\cal M}$ satisfying  the following properties:
\begin{enumerate}
\item [(P1)] $\mathcal D$ is an open set,   $K$ and $K_0$ are
compact sets,  $K$ is connected and
\begin{equation*}K_0\subset K\subset \mathcal D\subset {\overline {\mathcal D}}\subset {\cal M};\end{equation*}
\item[(P2)] let us set ${\mathcal F}$ to be the class of all continuos maps $\gamma:K\to {\cal D}$ with the property that there exists a continuos homotopy $\Gamma:[0,1]\times K\to {\cal D}$ such that:       $$\Gamma(0,\cdot)=id,\quad \Gamma(1,\cdot)=\gamma,\quad \Gamma(t,\bbm[\xi])=\bbm[\xi]\;\;\forall t\in [0,1],\,\forall \bbm[\xi]\in K_0;$$  we assume
\begin{equation}\label{mima}{ \Psi}^*:=\sup_{\gamma\in{\mathcal F}}\min_{\hbox{\scriptsize$\bbm[\xi]$}\in
K}\Psi(\gamma(\bbm[\xi]))<\min_{\hbox{\scriptsize$\bbm[\xi]$}\in K_0} \Psi(\bbm[\xi]);\end{equation} \item[(P3)] for every $\bbm[\xi]\in\partial\mathcal D$ such that $\Psi(\bbm[\xi])={\Psi}^*$, we
assume that $\partial \mathcal D$ is smooth at $\bbm[\xi]$ and
there exists a vector $\tau_{\hbox{\scriptsize$\bbm[\xi]$}}$ tangent to $\partial\mathcal
D$ at $\bbm[\xi]$ so that $\tau_{\hbox{\scriptsize$\bbm[\xi]$}}\cdot\nabla \Psi(\bbm[\xi])\neq 0$.

\end{enumerate}

\bigskip

Under these assumptions a critical point $\bbm[\xi]\in
{\mathcal D}$ of $\Psi$ with $\Psi(\bbm[\xi])=\Psi^*$
exists, as a standard deformation argument involving the gradient
flow of $\Psi$ shows. Moreover, since properties (P2)-(P3) continue to hold also for a function which is ${\cal C}^1$-close to $\Psi$, then  such critical point will  \textit{survive} small ${\cal C}^1$-perturbations.

 We define 
  \beq\label{didi}{\mathcal D}=\bigg\{\bbm[\xi]\in{\cal M}\;
\bigg|\,\Phi(\bbm[\xi]):=\frac{1}{2}\sum_{j=1}^NH(\xi_j,\xi_j)-\sum_{p\in Z}\frac{\alpha_p}{2}\sum_{j=1}^N G(\xi_j,p)-\frac12\sum_{j,k=1\atop j\neq k}^NG(\xi_j, \xi_k)  >-M \bigg\}\eeq 
where $M>0$ is a sufficiently large number yet to be chosen. By using the properties of the functions $H,\, G$ it is easy to check  that $\Phi$ satisfies \begin{equation}\label{coercivi}\Phi(\bbm[\xi])\to -\infty\hbox{ as }\bbm[\xi]\to \partial{\cal M},\end{equation} and this implies that  ${\mathcal
D}$ is compactly contained in ${\cal M}$.

\medskip

\subsection{Definition of $K$, $K_0$, and proof of (P1)}

In this section we will define the sets $K,\,K_0$ for which
properties (P1)-(P2) hold.

In the following we will  often use the complex numbers to identify the points in $\R^2$ and we will denote by ${\rm i}$ the imaginary unit.

First of all let us  fix  angles $\theta_p$ ($p\in Z$)  and  a number $\delta\in (0,\frac{\pi}{2})$ sufficiently small such that 
the cones  \beq\label{di1}\big\{p+\rho e^{{\rm i}(\theta_p+\theta)}\,\big|\, \rho\geq0, \,\theta \in [-\delta, \delta]\big\},\,\quad p\in Z\eeq are disjoint from one another.  We point out that such choice of angles always exists since $Z$ is finite. 
Possibly decreasing $\delta$, we may also assume \beq\label{di2}\di(p)>2\delta\;\;\forall p\in Z,\qquad |p-q|>4\delta\;\; \forall p,q\in Z,\,p\neq q.\eeq We recall that $\di(x)={\rm{dist}}(x,\partial\Omega)$.
According to assumption \eqref{akka1}, there exist $m$ points $p_1,\ldots, p_m\in Z$ ($m\geq 1$), such that we may split  $N=N_1+N_2+\ldots+ N_m$ with $N_r\in\N$ satisfying \beq\label{crik}1\leq N_r< 1+\alpha_r\quad \forall r.\eeq  Next we set $\{1,\ldots, N\}=I_1\cup \ldots\cup I_m$ where 
$$\begin{aligned}&I_1=\{1,2,\ldots, N_1\}, \\ &I_2=\{N_1+1,N_1+2,\ldots, N_1+N_2\},\\ &\ldots\\ &
I_r=\{N_1+\ldots +N_{r-1}+1,\ldots, N_1+\ldots+N_r\},\\ &\ldots
\\ &I_m=\{N_1+\ldots +N_{m-1}+1,\ldots, N\}.\end{aligned}$$
Next	lemma	establishes	the	unboundedness	of	the function $\Psi$ along suitable  $N$-tuple $\bbm[\xi]=(\xi_1,\ldots, \xi_N)$ whose components lie in the cones \eqref{di1}.

\begin{lemma}\label{prevenar} There exists a constant $C>0$ such that $$\Psi(\bbm[\xi])\leq C$$ for every $\bbm[\xi]=(\xi_1,\ldots, \xi_N)\in{\cal M}$  verifying
\beq\label{satti}\frac{\xi_j-p_r}{|\xi_j-p_r|}=e^{{\rm i}(\theta_{p_r}+j\frac{\delta}{N})}\;\;\;\forall j\in I_r,\;r=1,\ldots, m.\eeq
\end{lemma}
\begin{proof}
For any $\bbm[\xi] \in{\cal M}$ satisfying  \eqref{satti} 
 we get 
$$|\xi_j-\xi_k|\geq |\xi_j-p_r|\sin \frac{\delta}{2N}\quad \forall j,k\in I_r,\;j\neq k\;\;\;(r=1,\ldots, m).$$
By construction, for any $j\in I_r$  we have that $\xi_j$ belongs to the cone \eqref{di1} with $p=p_r$. This  implies that $$| \xi_j-\xi_k|\geq \mu\quad \forall j\in I_r,\, k\in I_s,\;r\neq s$$ where the value  $\mu$ depends only  on the choice of the angles $\theta_p$ and the number $\delta$.  Combining these facts with the properties of the functions $G$ and $H$
(see Appendix A),  we may  estimate
\beq\label{stimm}\begin{aligned}\Psi(\bbm[\xi])&
\leq- \sum_{p\in Z}\frac{\alpha_{p}}{4\pi} \sum_{j=1}^N\log\frac{1}{|\xi_j-p|}+\frac{1}{4\pi}\sum_{j,k=1\atop j\neq k}^N\log\frac{1}{|\xi_j-\xi_k|}+C\\ &\leq -
\frac{1}{4\pi}\sum_{r=1}^m\alpha_{p_r} \sum_{j=1}^N\log\frac{1}{|\xi_j-p_r|}+\frac{1}{4\pi}\sum_{r=1}^m\sum_{j,k\in I_r\atop j\neq k}\log\frac{1}{|\xi_j-\xi_k|} +C. 
\end{aligned}
\eeq
For a fixed $r\in\{1,\ldots, m\}$ and $j\in I_r$ we have
$$\begin{aligned}&-\alpha_{p_r}\log\frac{1}{|\xi_j-p_r|}+\sum_{k\in I_r\atop k\neq j}\log\frac{1}{|\xi_j-\xi_k|}\\ &\leq -\alpha_{p_r}\log\frac{1}{|\xi_j-p_r|}+(N_r-1)\log\frac{1}{|\xi_j-p_r|}-(N_r-1)\log \sin\frac{\delta}{2N}. 
\end{aligned}$$
Since $\alpha_{p_r}> N_r-1$ by \eqref{crik}, the above quantity is uniformly bounded above. Then the thesis follows by  \eqref{stimm}.

\end{proof}
Now we define $N$-tuple
$$\bbm[\xi]_0=(\xi_1^0,\ldots, \xi_N^0)$$  by
\beq\label{defxi0}\xi_j^0=p_r+\frac32\delta e^{{\rm i}(\theta_{p_r}+j\frac{\delta}{N})} \;\;\;\forall j\in I_r,\; r=1,\ldots, m.\eeq
In order to define $K$, we have to consider the $N$-tuple  $\bbm[\xi]=(\xi_1,\ldots, \xi_N)$ with the property that $N_r$ components  lie on the annulus  with radii $\delta$ and $2\delta$ centered in $p_r$. More precisely, setting
$$U_r:=\{\xi\in \R^2\,|\, \delta<|\xi-p_r|<2\delta\},$$ we introduce  the open set 
\beq\label{openset}\left\{\bbm[\xi]\in U_1^{N_1}\times \ldots \times U_m^{N_m} \,\Big|\,|\xi_j-\xi_k|>M^{-1}   \quad \forall j\neq k\right\}.\eeq
In principle, we do not know
whether \eqref{openset} is connected or not, so we will choose 
 a convenient connected component $W$. The choice of $\delta $ in \eqref{di1}-\eqref{di2} implies that $\xi_j^0\neq \xi_k^0$ for $j\neq k$, then $\bbm[\xi]_0$ belongs to \eqref{openset} provided that $M$ is sufficiently large. Now we are in conditions of defining $K$ and $K_0$: 
$$W:=\hbox{ the connected component of \eqref{openset} containing } \bbm[\xi]_0,$$
$$K:=\overline W,\quad K_0=\Big\{\bbm[\xi]\in K\,\Big|\, \min_{j\neq k}|\xi_j-\xi_k|=M^{-1}\Big\}.$$
$K$ is clearly connected and $K_0\subset\partial W\subset K$. Moreover by construction,  since  the closed annulii $\overline{U}_r$ are contained in $\Omega\setminus Z$ and disjoint from one another  according to  \eqref{di2}, using the properties of the $G$ and $H$ (see Appendix A) we get that   the functions $\sum_{j\neq k}H(\xi_j, \xi_k)$,  $\sum_{j}H(\xi_j, \xi_j)$ and  $\sum_{p\in Z}\alpha_p\sum_{j} G(\xi_j, p)$ are uniformly bounded in the set $U_1^{N_1}\times\ldots\times U_m^{N_m}$, therefore 
\beq\label{iuni}\sum_{j,k=1\atop j\neq k}^NH(\xi_j, \xi_k),\;\;\sum_{j=1}^NH(\xi_j, \xi_j),\;\; \sum_{p\in Z}\alpha_p\sum_{j=1}^N G(\xi_j, p)=O(1)\;\hbox{  in }K\eeq with the above quantity $O(1)$ uniformly bounded independently of $M$. 
On the other hand in the set $K$ we have $G(\xi_j,\xi_k)\leq \frac{1}{2\pi}\log M+C$ for $j\neq k$ by \eqref{matis}. Consequently for large $M$ we also have $K\subset {\cal D}$.

\subsection{Proof of (P2)}

The definition of the max-min value $\Psi^*$ in \eqref{mima} depends on the particular $M>0$ chosen in \eqref{didi}. To emphasize this fact we denote this max-min value by $\Psi^*_M$. In this section we will prove that (P2) holds for $M$ sufficiently large. To this aim we need the estimate for $\Psi^*_M$ provided by the following proposition.
\begin{proposition}\label{loweresti}
The quantity $\Psi^*_M$ is bounded independently  of the large number $M$ used to define ${\cal D}$, namely there exist two constants $c, C$ independent of $M$ such that 
\begin{equation}\label{rox}c\leq\Psi^*_M
\leq C.\end{equation}
\end{proposition}
\begin{proof}
Let $\gamma\in {\mathcal F},$ namely $\gamma:K\to {\cal D}$ is a continuous map such that  there exists a continuos homotopy $\Gamma:[0,1]\times K\to {\cal D}$ satisfying:       $$\Gamma(0,\cdot)=id,\quad \Gamma(1,\cdot)=\gamma,\quad \Gamma(t,\bbm[\xi])=\bbm[\xi]\;\;\forall t\in [0,1],\,\forall \bbm[\xi]\in K_0.$$ 
Let $\gamma=(\gamma_1,\ldots, \gamma_N) $ and  $\Gamma=(\Gamma_1,\ldots, \Gamma_N)$ with $\gamma_j:K\to \R^2$ and $\Gamma_j: [0,1]\times K\to \R^{2}$. We now define a new homotopy
${\cal H}:[0,1]\times K\to \R^{2N}$ with components ${\cal H}_j: [0,1]\times K\to \R^{2}$:
\beq\label{citi}{\cal H}_j(t,\bbm[\xi])=p_r+|\xi_j-p_r|\cdot \frac{\Gamma_j(t, \bbm[\xi])-p_r}{|\Gamma_j(t, \bbm[\xi])-p_r|}\;\;\;\forall j\in I_r,\;r=1,\ldots, m.\eeq
Clearly ${\cal H}$ is a continuous map. 
By $\Gamma(0,\cdot)=id$ we immediately get 
$${\cal H}(0,\cdot)=id.$$
We claim that \beq\label{pugli}{\cal H}(t, \partial W)\subset \partial W\quad \forall t\in [0,1].\eeq
Indeed, if $\bbm[\xi]\in K_0$, then $\Gamma_j(t, \bbm[\xi])=\xi_j$ and consequently, by definition, ${\cal H}_j(t,\bbm[\xi])=\xi_j$, which implies 
${\cal H}(t, \bbm[\xi])=\bbm[\xi]\in K_0\subset \partial W$.
On the other hand, if $\bbm[\xi]\in \partial W\setminus K_0$, we get $\bbm[\xi]\in \partial (U_1^{N_1}\times \ldots \times U_m^{N_m} )$. Then there exists $j$ such that $\xi_j\in \partial U_j$. Assume, for instance,  $\xi_1\in \partial U_1$, namely either $|\xi_1-p_1|=\delta$ or $|\xi_1-p_1|=2\delta$. By definition we have $|{\cal H}_1(t,\bbm[\xi])-p_1|=|\xi_1-p_1|\in \{\delta, 2\delta\}$, which implies  that ${\cal H}_1(t,\bbm[\xi])\in\partial U_1$ and then 
${\cal H}(t,\bbm[\xi])\in \partial W$. Hence \eqref{pugli} follows. 

We have thus proved that  the homotopy ${\cal H}$ applies $\partial W$ into itself. The theory of the topological degree gives that if $\bbm[\xi]\in W$ then $\deg({\cal H}(1, \cdot), W, \bbm[\xi])=\deg({\cal H}(0, \cdot), W, \bbm[\xi])=\deg (id, W, \bbm[\xi])=1$. 
Then there exists $\bbm[\xi]_\gamma\in W$ such that ${\cal H}(1, \bbm[\xi]_\gamma)=\bbm[\xi]_0$, namely, by \eqref{citi}, setting $\bbm[\xi]_\gamma=(\xi_1^\gamma,\ldots, \xi_N^\gamma)$, 
$$\gamma_j( \bbm[\xi]_\gamma)-p_r=\frac{|\gamma_j( \bbm[\xi]_\gamma)-p_r|}{|\xi_j^\gamma-p_r|}(\xi_j^0-p_r)\;\;\;\forall j\in I_r,\;r=1,\ldots, m.$$
Using \eqref{defxi0} we get
$$\frac{\gamma_j( \bbm[\xi]_\gamma)-p_r}{|\gamma_j( \bbm[\xi]_\gamma)-p_r|}=e^{{\rm i}(\theta_{p_r}+j\frac{\delta}{N})}\;\;\;\forall j\in I_r,\;r=1,\ldots, m.$$
In such a situation we deduce that  $\gamma(\bbm[\xi])$ has the form \eqref{satti}. So Lemma \ref{prevenar} applies and gives 
$$\min_{\hbox{\scriptsize$\bbm[\xi]$}\in K}\Psi(\gamma(\xi))\leq \Psi(\gamma(\bbm[\xi]_\gamma))\leq C$$
and  the constant on the right hand side depends only  on the choice of the angles $\theta_p$ and the number  $\delta$ in \eqref{di1}-\eqref{di2}; in particular, $C$ is independent  of $\gamma$ and $M$. 
 Hence,  by taking the supremum for all the maps $\gamma\in{\cal F}$,   we conclude that the max-min value $\Psi^*$ is bounded above  independently of $M$, as desired.

Finally, in order to prove the lower boundedness, by taking $\eta=id$ in the definition \eqref{mima} we get
$$\Psi^*_M\geq \min_{\hbox{\scriptsize$\bbm[\xi]$}\in K}\Psi(\bbm[\xi])\geq \min_{\hbox{\scriptsize$\bbm[\xi]$}\in K}\bigg(\frac{1}{2}\sum_{j=1}^N H(\xi_j,\xi_j)-\sum_{p\in Z}\frac{\alpha_p}{2}\sum_{j=1}^N G(\xi_j, p)\bigg) .$$ As we have already observed in \eqref{iuni},   the function in the bracket is uniformly bounded 
in
the set $K$
independently of $M$.
\end{proof}

Then the max-min inequality (P2) will follow once we have proved the next result. \begin{proposition}\label{butta} The following holds:\begin{equation}\label{rocco}\min_{\hbox{\scriptsize$\bbm[\xi]$}\in K_0}\Psi(\bbm[\xi])=\min\left\{\Psi(\bbm[\xi])\,\Big|\,\bbm[\xi]\in K,\, \min_{j\neq k}|\xi_j-\xi_k|=M^{-1}\right\} \to +\infty \hbox{ as }M\to +\infty.\end{equation}\end{proposition}

\begin{proof}
Let $\bbm[\xi]_n=(\xi_1^n,\ldots,\xi_N^n)\in K$ be such that $\min_{j\neq k}|\xi_j^n-\xi_k^n| \to 0$ as $n\to +\infty$. Possibly passing to a subsequence, we may assume 
\begin{equation}\label{cover}|\xi_{j_0}^n-\xi_{k_0}^n| \to 0\hbox{ as }n\to +\infty\end{equation} for some  $j_0\neq k_0$. So, by using \eqref{iuni},  we may estimate
$$\Psi(\bbm[\xi]_n)=\frac{1}{4\pi}\sum_{j,k=1\atop j\neq k}^N\log\frac{1}{|\xi_{j}^n-\xi_{k}^n|}+O(1) \geq \frac{1}{2\pi}\log\frac{1}{|\xi_{j_0}^n-\xi_{k_0}^n|}+O(1)\to +\infty.$$
\end{proof}

\subsection{Proof of (P3)}
We shall show that  (P3) holds provided that $M$ is sufficiently large. 
A key role for the validity of the compactness property (P3) is played by the following lemma. 
\begin{lemma}\label{assoassu} Let $\bbm[\xi]=(\xi_1^n,\ldots, \xi_N^n)\in {\cal M}$ be such that \beq\label{boupsiq}\Psi(\bbm[\xi]_n)=O(1),\quad \min_{j\neq k}|\xi_j^n-\xi_k^n|\to 0.\eeq 
Then there exist $j_0\neq k_0$ such that, possibly passing to a subsequence,  
the following holds: $$|\xi_{j_0}^n-\xi_{k_0}^n|=o(\di(\xi_{j_0}^n)).$$
\end{lemma}
\begin{proof}
Suppose by contradiction that, up to a subsequence,  $|\xi_j^n-\xi_k^n|\geq c\di(\xi_j^n)$ for all $j\neq k$. Then, according to Corollary \ref{robin1} we have 
$G(\xi_j^n, \xi_k^n)=O(1)$ for all $j\neq k$.
The hypothesis \eqref{boupsiq} gives 
$$\min_{j=1,\ldots,N}\di(\xi_j^n)\to 0.$$
  So we may evaluate
$$\Psi(\bbm[\xi]_n)=\frac12\sum_{j=1}^NH(\xi_j^n,\xi_j^n)-\frac12\sum_{p\in Z}\alpha_p\sum_{j=1}^NG(\xi_j^n,p) +O(1)\leq \frac12\sum_{j=1}^NH(\xi_j^n,\xi_j^n)+O(1)\to-\infty $$ by \eqref{matis}-\eqref{coercive}, 
in contradiction with \eqref{boupsiq}.
\end{proof}
By Proposition \ref{loweresti}  we get $\Psi^*=\Psi_M^*=O(1)$ as $M\to +\infty$. Then  (P3) will follow once we have proved the assertion of tangential derivative being non-zero over the boundary of ${\cal D}$ for uniformly bounded values of $\Psi$ provided that $M$ is large enough. We proceed by contradiction: assume that there exist $\bbm[\xi]_n=(\xi_1^n, \ldots, \xi_N^n)\in {\cal M}$ and a vector $(\beta^n_1, \beta_2^n)\neq (0,0)$ such that 
$$\min_{j\neq k}|\xi_j^n-\xi_k^n|=o(1),$$ 
\beq\label{boupsi}\Psi(\bbm[\xi]_n)=O(1),\eeq
\beq\label{exp}\beta_1^n\nabla \Psi(\bbm[\xi]_n)+\beta_2^n\nabla \Phi(\bbm[\xi]_n)=0.\eeq
Last expression can be read as $\nabla \Psi(\bbm[\xi]_n)$ and $\nabla \Phi(\bbm[\xi]_n)$ are linearly dependent. Observe that, according to the Lagrange multiplier Theorem,  this contradicts either the smoothness of $\partial {\cal D}$ or the nondegeneracy of $\nabla \Phi(\bbm[\xi]_n)$ on the tangent space at the level $\Psi^*$.
Without loss of generality we may assume \beq\label{ops}(\beta_1^n)^2+(\beta_2^n)^2=1\;\hbox{ and } \;\beta_1^n+\beta_2^n\geq 0.\eeq
The identities $\beta_1^n\nabla_{\xi_j}\Psi(\bbm[\xi]_n)+\beta_2^n\nabla_{\xi_j}\Phi(\bbm[\xi]_n)=0$ imply
\beq\label{leone}(\beta_1^n+\beta_2^n)\nabla_x H(\xi_j^n, \xi_j^n)-\frac{\beta_1^n+\beta_2^n}{2}\sum_{p\in Z}\alpha_p\nabla_x G(\xi_j^n,p)+(\beta_1^n-\beta_2^n)\sum_{k=1\atop k\neq j}^N\nabla_x G(\xi_j^n,\xi_k^n)=0\quad \forall j.\eeq
The object of the remaining part of the section is to expand the left hand side of \eqref{exp} and to prove that the leading term is not zero, so that the contradiction arises. 
To this aim we distinguish five cases which will all lead to a contradiction. 

In what follows at many steps of the arguments we will pass to a subsequence, without further notice.

\bigskip

\noindent {\bf Case 1.} \textit{$\beta_1^n+\beta_2^n\geq c>0$, $\beta_1^n-\beta_2^n\geq o(1)$ and there exists $\hat{\jmath}\in \{1,\ldots,N\}$ such that $\di(\xi_{\hat\jmath}^n)\to 0$.}

\bigskip
Assume that $\xi_{\hat\jmath}^n\to\xi_{\hat\jmath}\in\partial\Omega$. We can split $\{1,\ldots,N\}=I\cup J$ where 
$$ I=\{j\,|\, \xi_j^n\to \xi_{\hat\jmath}\},\qquad J=\{j\,|\, |\xi_j^n-\xi_{\hat\jmath}|\geq c\}.$$
By using the notation of Lemma \ref{robin}, we multiply the identity \eqref{leone} by $\nu_{\xi_j^n}$ and, adding in $j\in I $,  we obtain 
\beq\label{fiu}(\beta_1^n+\beta_2^n)\sum_{j\in I}\frac{\xi_j^n-\bar{\xi_j^n}}{|\bar{\xi_j^n}-\xi_j^n|^2}\nu_{\xi_j^n}+(\beta_1^n-\beta_2^n)\sum_{j,k\in I\atop j\neq k}\bigg(\frac{\xi_k^n-\xi_j^n}{|\xi_j^n-\xi_k^n|^2}+\frac{\xi_k^n-\bar{\xi_j^n}}{|\bar{\xi_j^n}-\xi_k^n|^2}\bigg)\nu_{\xi_j^n}=O(1).\eeq
We now estimate each term of the above sum in order to get a contradiction.
First observe that $\xi-\bar{\xi}=2\di(\xi)\nu_{\xi}$,
 therefore 
\beq\label{gui1}\frac{\xi_j^n-\bar{\xi_j^n}}{|\bar{\xi_j^n}-\xi_j^n|^2}\nu_{\xi_j^n}=
\frac{1}{2\di(\xi_j^n)}\quad \forall j\in I.\eeq
On the other hand, since $\partial\Omega$ is of class    ${\cal C}^2$, for all $j,k\in I$ with $j\neq k$ we get $$\nu_{\xi_j^{n}}-\nu_{\xi_{k}^n}=O(|\xi_j^n-\xi_k^n|)$$
by which we deduce 
\beq\label{gui2}\sum_{j,k\in I\atop j\neq k}\frac{\xi_k^n-\xi_j^n}{|\xi_j^n-\xi_k^n|^2}\nu_{\xi_j^n}=\sum_{j,k\in I\atop j<k}\frac{\xi_k^n-\xi_j^n}{|\xi_j^n-\xi_k^n|^2}(\nu_{\xi_j^n}-\nu_{\xi_k^n})=O(1).\eeq
Moreover Corollary \ref{robin2} yields
\beq\label{gui3}\frac{\xi_k^n-\bar{\xi_j^n}}{|\bar{\xi_j^n}-\xi_k^n|^2}\nu_{\xi_j^n}=\frac{\di(\xi_j^n)+\di(\xi_k^n)}{|\bar{\xi}_j^n-\xi_k^n|^2}+O(1)\quad \forall j,k\in I,\, j\neq k.\eeq
By inserting the  estimates \eqref{gui1}, \eqref{gui2} and  \eqref{gui3} in \eqref{fiu} and using the assumptions $\beta_1^n+\beta_2^n\geq c>0$,  $\beta_1^n-\beta_2^n\geq o(1)$, we arrive at 
$$c\sum_{j\in I}\frac{1}{2\di(\xi_j^n)}+\sum_{j,k\in I\atop j\neq k}o(1)\frac{\di(\xi_j^n)+\di(\xi_k^n)}{|\bar{\xi_j^n}-\xi_k^n|^2}\leq O(1).$$ Taking into account of the obvious inequalities $|\bar{\xi_j^n}-\xi_k^n|\geq \di(\xi_j^n),\,\di(\xi_k^n)$, we deduce 
$$c\sum_{j\in I}\frac{1}{2\di(\xi_j^n)} +\sum_{j\in I}\frac{o(1)}{\di(\xi_j^n)}\leq O(1)$$  which is a contradiction. 

\bigskip

\noindent {\bf Case 2.}  \textit{$\beta_1^n-\beta_2^n\to 0$.}

\bigskip

 According to \eqref{ops} we have \beq\label{crib}\beta_1^n=\frac{\sqrt{2}}{2}+o(1),\qquad\beta_2^n=\frac{\sqrt{2}}{2}+o(1).\eeq Hence, in order to avoid case 1, we assume
\beq\label{asd}\di(\xi_j^n)\geq c>0\quad \forall j=1,\ldots,N.\eeq
Now we distinguish two cases:\begin{itemize}
\item[(a)] there exists $p\in Z$ and $j\in\{1,\ldots, N\}$ such that $\xi_j^n\to p$;
\item[(b)] for every $j$ we have $\xi_j^n\to\xi_j\in \Omega\setminus Z$.
 \end{itemize}
 
 First assume (a). We can split $\{1,\ldots,N\}=I\cup J$ where$$ I=\{j\,|\, \xi_j^n\to p\},\quad J=\{j\,|\, |\xi_j^n-p|\geq c\}.$$
Using Remark \ref{remrem}   the identities \eqref{leone} give
$$(\beta_1^n+\beta_2^n)\frac{\alpha_p}{2}\frac{\xi_j^n-p}{|\xi_j^n-p|^2}+\sum_{k\in I\atop k\neq j}(\beta_1^n-\beta_2^n)\frac{\xi_k^n-\xi_j^n}{|\xi_j^n-\xi_k^n|^2}=O(1)\quad \forall j\in I.$$
Let us multiply the above identity by $\xi_j^n-p$ and next sum in $j\in I$; using the following
\beq\label{idi}\sum_{j,k\in I\atop j\neq k}\frac{\xi_j^n-\xi_k^n}{|\xi_j^n-\xi_k^n|^2}(\xi_j^n-p)=\sum_{j,k\in I\atop j< k}1=\frac{\#I(\# I-1)}{2},\eeq where $\#I$ denotes the number of elements of $I$, 
we obtain
$$(\beta_1^n+\beta_2^n)\frac{\alpha_p}{2}\# I=
(\beta_1^n -\beta_2^n)\frac{\# I(\#I-1)}{2}+o(1).$$ 
The contradiction arises because of \eqref{crib}.

Now consider case (b).  According to Lemma \ref{assoassu} there exist $j_0\neq k_0$ such that $\xi_{j_0}=\xi_{k_0}$.   
  Taking into account of \eqref{asd}, we may evaluate
$$\Psi(\bbm[\xi]_n)= \frac{1}{4\pi}\sum_{j,k=1\atop j\neq k}^N\log\frac{1}{|\xi_j^n-\xi_k^n|}+O(1)\geq \frac{1}{2\pi}\log\frac{1}{|\xi_{j_0}^n-\xi_{k_0}^n|}+O(1)\to +\infty.$$ 

\bigskip

\noindent {\bf Case 3.} \textit{There exist  $\hat{\jmath}\neq\hat{k}$,  such that $\xi_{\hat{\jmath}}^n,\,\xi_{\hat{k}}^n\to\xi\in\Omega$ and $|\xi_{\hat{\jmath}}^n-\xi_{\hat{k}}^n|=o(|\xi_{\hat{\jmath}}^n-p|)$ for all $p\in Z$. 
}

\bigskip

We may split $\{1,\ldots, N\}=I\cup J$ where 
$$I=\big\{j\,\big|\, |\xi_{j}^n-\xi_{\hat{\jmath}}^n|\leq C|\xi_{\hat{\jmath}}^n-\xi_{\hat{k}}^n|\big\},\quad J=\big\{j\,\big|\,  |\xi_{\hat{\jmath}}^n-\xi_{\hat{k}}^n|=o(|\xi_j^n-\xi_{\hat{\jmath}}^n|)\big\}.$$
We observe that by construction $|\xi_{\hat{\jmath}}^n-\xi_{\hat{k}}^n|=o(|\xi_j^n-\xi_k^n|)$ for all $j\in I$ and $k\in J$, by which $$\nabla_xG(\xi_j^n, \xi_k^n)=o\bigg(\frac{1}{|\xi_{\hat{\jmath}}^n-\xi_{\hat{k}}^n|}\bigg)\quad \forall j\in I,\, k\in J.$$ Moreover, by hypothesis,  $|\xi_{\hat{\jmath}}^n-\xi_{\hat{k}}^n|=o(|\xi_j^n-p|)$ for all $j\in J$ and $p\in Z$, which implies  $$\nabla_xG(\xi_j^n, p)=o\bigg(\frac{1}{|\xi_{\hat{\jmath}}^n-\xi_{\hat{k}}^n|}\bigg)\quad \forall j\in I,\,p\in Z.$$
Then for any $j\in I$ the identity \eqref{leone} gives 
\beq\label{qqw10}(\beta_1^n-\beta_2^n)\sum_{k\in I\atop k\neq j}\frac{\xi_k^n-\xi_j^n}{|\xi_j^n-\xi_k^n|^2}=o\bigg(\frac{1}{|\xi_{\hat{\jmath}}^n-\xi_{\hat{k}}^n|}\bigg)\quad \forall j\in I.\eeq So we multiply \eqref{qqw10} by $\xi_j^n-\xi_{\hat{\jmath}}^n$ and sum in $j\in I$: by using the identity \eqref{idi} (with $\xi_{\hat{\jmath}}^n$ in the place of $p$), we obtain 
$$(\beta_1^n-\beta_2^n)\frac{\#I(\#I-1)}{2}=o(1) .$$ Taking into account that $\#I\geq 2$ since  $\hat{\jmath},\, \hat{k}\in I$, we deduce  $\beta_1^n-\beta_2^n=o(1)$ and we fall again in the previous case.
\bigskip

\noindent {\bf Case 4.} \textit{There exist $p\in Z$ and $\hat{\jmath}\neq \hat{k}$ such that $\xi_{\hat{\jmath}}^n,\, \xi_{\hat{k}}^n\to p$. 
}

\bigskip

Let us set $$I_p=\{j\,|\, \xi_j^n\to p\}\neq \emptyset.$$ We will split $I$ into $\ell$ pieces, with $\ell\geq 1$,
$$I_p=I_1\cup\ldots \cup I_\ell,\quad I_r\neq\emptyset,\;r=1,\ldots,\ell,$$ in such a way that
\beq\label{mario1}|\xi_j^n-p|\sim |\xi_k^n-p|\quad \forall j,k\in I_r\eeq and 
\beq\label{mario2}|\xi_j^n-p|=o(|\xi_k^n-p|)\quad \forall j\in I_r, \, \forall k\in I_{r+1}\cup\ldots\cup I_\ell.\eeq Here we use the notation $\sim$ to denote sequences which in the limit $n\to +\infty$ are of the same order. 
By construction we clearly have $|\xi_j^n-\xi_k^n|\sim |\xi_k^n-p|$ for all $j\in I_r$, $k\in I_{r+1}\cup\ldots\cup I_\ell$. On the other hand, in order to not fall again in the previous case, we may assume $|\xi_j^n-\xi_k^n|\sim  |\xi_k^n-p|$ for all $j, k\in I_r$. So we have
\beq\label{mario3}|\xi_j^n-\xi_k^n|\sim |\xi_k^n-p|\quad \forall j\in I_r,\;\;\forall k\in I_{r}\cup\ldots\cup I_{\ell} .\eeq
Combining \eqref{mario2}-\eqref{mario3} we get
\beq\label{mario33}\nabla_xG(\xi_j^n, \xi_k^n)=o\bigg(\frac{1}{|\xi_j^n-p|}\bigg)\quad \forall j\in I_r, \,\,\forall k\in I_{r+1}\cup\ldots\cup I_{\ell}.\eeq
Next consider $r\in\{1,\ldots,\ell\}$ and we write \eqref{leone} for $j\in I_r$: using \eqref{mario33}, we obtain 
\beq\label{qw2}(\beta_1^n-\beta_2^n)\sum_{k\in I_1\cup\ldots \cup I_r\atop k\neq j}\frac{\xi_j^n-\xi_k^n}{|\xi_j^n-\xi_k^n|^2}=(\beta_1^n+\beta_2^n)\frac{\alpha_p}{2}\frac{\xi_j^n-p}{|\xi_j^n-p|^{2}}+o\bigg(\frac{1}{|\xi_j^n-p|}\bigg)\quad \forall j\in I_r.\eeq 
By  \eqref{mario2} and \eqref{mario3} we find 
$|\xi_j^n-p|=o(|\xi_j^n-\xi_k^n|) $ $ j\in I_1\cup\ldots \cup I_{r-1}$ and  $k\in I_r$; consequently, by interchanging the roles of $j$ and $k$,  
we compute 
 \beq\label{obs}\frac{(\xi_j^n-\xi_k^n)(\xi_j^n-p)}{|\xi_j^n-\xi_k^n|^2}=1+\frac{(\xi_j^n-\xi_k^n)(\xi_k^n-p)}{|\xi_j^n-\xi_k^n|^2}=1+o(1)\quad \forall j\in I_r,\;\;\forall k\in I_1\cup\ldots \cup I_{r-1}.\eeq 
Now we multiply \eqref{qw2} by $\xi_j^n-p$ and sum in $j\in I_r$. By reasoning as in \eqref{idi} we get 
\beq\label{gra1}(\beta_1^n-\beta_2^n)\frac{\#I_1(\#I_1-1)}{2}=
(\beta_1^n+\beta_2^n)\frac{\alpha_p}{2}\#I_1+o(1),\eeq  for $r=1$ and, using \eqref{obs},   
\beq\label{gra2}(\beta_1^n-\beta_2^n)\bigg(\frac{\#I_r(\#I_r-1)}{2}+\#I_r(\#I_1+\ldots +\#I_{r-1})\bigg)=(\beta_1^n+\beta_2^n)\frac{\alpha_p}{2}\#I_r+o(1)\quad \forall r>1.\eeq
 Let us consider the identity \eqref{gra2} for $r=\ell$: taking into account that $\sum_{r=1}^\ell\#I_r=\#I_p\geq 2$ since $\hat{\jmath},\,\hat{k}\in I_p$, we find that both $\beta_1^n-\beta_2^n$ and $\beta_1^n+\beta_2^n$ are multiplied by positive quantity.
 Then,  using that at least one between $\beta_1^n-\beta_2^n$ and $\beta_1^n+\beta_2^n$ does not go to zero because of $(\beta_1^n)^2+(\beta_2^n)^2=1$,
 and recalling \eqref{ops},  we deduce
 \beq\label{bell}\beta_1^n-\beta_2^n,\, \beta_1^n+\beta_2^n\geq c>0.\eeq
 Furthermore, we also have 
 $$|\beta_1^n|,\, |\beta_2^n|\geq c>0.$$ Otherwise, if $\beta_1^n=o(1)$, by \eqref{ops}  we would obtain $\beta_2^n=1+o(1)$ and this contradicts \eqref{gra1}-\eqref{gra2}. On the other hand,   if $\beta_2^n=o(1)$, then $\beta_1^n=1+o(1)$ and, consequently, $\#I_1-1=\alpha_p$  by \eqref{gra1}, in contradiction with the assumption \eqref{akka2}. So we may set
 \beq\label{gra3}\frac{\beta_2^n}{\beta_1^n}=a(1+o(1))\hbox{ for some } a\neq 0.\eeq

Now we are going to evaluate separate pieces of the energy.
We begin with the following: 
$$\begin{aligned}&-\frac{\alpha_p}{2}\sum_{j\in I_p}G(\xi_j^n,p)+\frac12\sum_{j,k\in I_p\atop j\neq k} G(\xi_j^n,\xi_k^n)\\ &=-\frac{\alpha_p}{2}\sum_{r=1}^\ell \sum_{j\in I_r}G(\xi_j^n,p)+\frac12\sum_{r=1}^\ell\sum_{j,k\in I_r\atop j\neq k} G(\xi_j^n,\xi_k^n)+\sum_{r=1}^\ell\sum_{j\in I_r\atop k\in I_1\cup\ldots\cup I_{r-1}}G(\xi_j^n,\xi_k^n)\\ &=\sum_{r=1}^\ell \bigg(-\frac{\alpha_p}{2}\sum_{j\in I_r}\log\frac{1}{|\xi_j^n-p|}+\frac12\sum_{j,k\in I_r\atop j\neq k} \log\frac{1}{|\xi_j^n-\xi_k^n|}+\sum_{j\in I_r\atop k\in I_1\cup\ldots\cup I_{r-1}}\log\frac{1}{|\xi_j^n-\xi_k^n|}\bigg)+O(1).\end{aligned} $$
\eqref{mario3} implies  that, fixed $r$, all the the terms $ |\xi_j^n-\xi_k^n|$ and $|\xi_k^n-p|$ with  $j\in I_1\cup\ldots \cup I_{r}$ and $k\in I_r$  are of the same order. 
Then, by interchanging the roles of $j$ and $k$, we have $ |\xi_j^n-\xi_k^n|\sim |\xi_j^n-p|$ for all $j\in I_r$ and $k\in I_1\cup\ldots \cup I_{r}$.
Therefore, fixed $j_r\in I_r$ arbitrarily, we may write $$\begin{aligned}&-\frac{\alpha_p}{2}\sum_{j\in I_r}\log\frac{1}{|\xi_j^n-p|}+\frac12\sum_{j,k\in I_r\atop j\neq k} \log\frac{1}{|\xi_j^n-\xi_k^n|}+\sum_{j\in I_r\atop k\in I_1\cup\ldots\cup I_{r-1}}\log\frac{1}{|\xi_j^n-\xi_k^n|}\\ &=  \bigg(-\frac{\alpha_p}{2}\sum_{j\in I_r}1+\frac12\sum_{j,k\in I_r\atop j\neq k} 1+\sum_{j\in I_r\atop k\in I_1\cup\ldots\cup I_{r-1}}1\bigg)\log\frac{1}{|\xi_{j_r}^n-p|}+O(1)\\ &= \Big(-\frac{\alpha_p}{2}\#I_r+\frac{\#I_r(\#I_r-1)}{2}+ \#I_r(\#I_1+\ldots +\#I_{r-1})\Big)\log\frac{1}{|\xi_{j_r}^n-p|} +O(1)\\ &= (a+o(1))\Big(\frac{\alpha_p}{2}\#I_r+\frac{\#I_r(\#I_r-1)}{2}+ \#I_r(\#I_1+\ldots +\#I_{r-1})\Big)\log\frac{1}{|\xi_{j_r}^n-p|} +O(1)\end{aligned}$$
by \eqref{gra2} and \eqref{gra3}, and the last term goes to $+\infty$ if $a>0$ and $-\infty$ if $a<0$.
We have thus proved that \beq\label{puff}\frac{1}{a}\bigg(-\frac{\alpha_p}{2}\sum_{j\in I_p}G(\xi_j^n,p)+\frac12\sum_{j,k\in I_p\atop j\neq k} G(\xi_j^n,\xi_k^n)\bigg)\to +\infty\quad\forall p\in Z\hbox{ such that } I_p\neq \emptyset.\eeq In order to conclude it is sufficient to sum up all the previous information in order to estimate  the total $\Psi(\bbm[\xi]_n)$. Indeed, by \eqref{bell}, in order to not fall again in the case 1, we may assume
$$\di(\xi_j^n)\geq c>0\quad \forall j=1,\ldots, N.$$ Moreover, in order to avoid case 3, we assume $|\xi_j^n-\xi_k^n|\geq c>0$ if $(j,k)\not\in I_p\times I_p$ for any $p$, by which
$$G(\xi_j^n,\xi_k^n)=O(1)\quad \forall (j,k)\not\in \bigcup_{p\in Z} (I_p\times I_p).$$
Then we may evaluate
$$\begin{aligned}\Psi(\bbm[\xi]_n)=&-\sum_{p\in Z}\frac{\alpha_p}{2}\sum_{j=1}^NG(\xi_j^n,p)+\frac12\sum_{j,k=1\atop j\neq k}^N G(\xi_j^n, \xi_k^n)+O(1)\\ &=-\sum_{p\in Z }\frac{\alpha_p}{2}\sum_{j\in I_p}G(\xi_j^n,p)+\frac12\sum_{p\in Z}\sum_{j,k\in I_p\atop j\neq k} G(\xi_j^n, \xi_k^n)+O(1)\end{aligned}$$
 and  \eqref{puff} allows us to conclude 
 $$\frac{1}{a}\Psi(\bbm[\xi]_n)\to +\infty$$ in contradiction with \eqref{boupsi}.

\bigskip

\noindent {\bf Case 5.} \textit{Conclusion}.

\bigskip

According to Lemma \ref{assoassu} there exist $j_0\neq k_0$ such that  $$\xi_{j_0}^n,\;\xi_{k_0}^n\to \xi\in\overline{\Omega}$$ and $$|\xi_{j_0}^n-\xi_{k_0}^n|=o(\di(\xi_{j_0}^n)).$$ If $\xi\in Z$, then we are in the case 4. So let us assume $\xi\in \overline{\Omega}\setminus Z$.
We may split $\{1,\ldots, N\}=I\cup J$ where 
$$I=\{j\,|\, |\xi_j^n-\xi_{j_0}^n|\leq C|\xi_{j_0}^n-\xi_{k_0}^n|\},\quad J=\{j\,|\,  |\xi_{j_0}^n-\xi_{k_0}^n|=o(|\xi_j^n-\xi_{j_0}^n|)\}.$$
By construction we deduce 
\beq\label{pneu}|\xi_j^n-\xi_{j_0}^n|=o(\di(\xi_j^n))\quad \forall j\in I.\eeq Then, a direct application of Remark  \ref{remrem} yields
$$\nabla_x H(\xi_j^n, \xi_j^n)=O\bigg(\frac{1}{\di(\xi_j^n)}\bigg)=o\bigg(\frac{1}{|\xi_j^n-\xi_{j_0}^n|}\bigg)\quad \forall j\in I,$$ and, similarly, $$\nabla_xG(\xi_j^n, p)=o\bigg(\frac{1}{|\xi_{j}^n-\xi_{j_0}^n|}\bigg)\quad \forall j\in I,\;\forall p\in Z,$$

$$\nabla_x G(\xi_j^n,\xi_k^n)=
\frac{1}{2\pi}\frac{\xi_k^n-\xi_j^n}{|\xi_j^n-\xi_k^n|^2}+o\bigg(\frac{1}{|\xi_j^n-\xi_{j_0}^n|}\bigg)\quad \forall j,k\in I,\;j\neq k.$$
On the other hand by construction we also get \beq\label{pneu1}|\xi_{j}^n-\xi_{j_0}^n|=o(|\xi_j^n-\xi_k^n|)\quad \forall j\in I,\;\forall k\in J.\eeq  Then, by \eqref{pneu} and \eqref{pneu1},  using again Remark \ref{remrem}, 

 $$\nabla_x G(\xi_j^n,\xi_k^n)=\frac{1}{2\pi}\frac{\xi_k^n-\xi_j^n}{|\xi_j^n-\xi_k^n|^2}+O\bigg(\frac{1}{\di(\xi_j^n)}\bigg)=o\bigg(\frac{1}{|\xi_j^n-\xi_{j_0}^n|}\bigg)\quad \forall j\in I,\;\forall k\in J.$$
Summing up all the above estimates,  for any $j\in I$ the identity \eqref{leone} gives 
\beq\label{qqw1}(\beta_1^n-\beta_2^n)\sum_{k\in I\atop k\neq j}\frac{\xi_k^n-\xi_j^n}{|\xi_j^n-\xi_k^n|^2}=o\bigg(\frac{1}{|\xi_{j}^n-\xi_{j_0}^n|}\bigg)\quad \forall j\in I.\eeq Now we multiply \eqref{qqw1} by $\xi_j^n-\xi_{j_0}^n$ and sum in $j\in I$: by using the identity \eqref{idi} (which holds with  $\xi_{j_0}^n$ in the place of $p$), we obtain 
$$(\beta_1^n-\beta_2^n)\frac{\#I(\#I-1)}{2}=o(1) .$$ Taking into account that $\#I\geq 2$ since  $j_0,\,k_0\in I$, we deduce  $\beta_1^n-\beta_2^n=o(1)$ and we are back in the case 2.

\begin{remark}\label{compac1} The major problem in proving the validity of the property (P3) has been to exclude the possibility that two or more components of $(\xi_1^n,\ldots, \xi_N^n)$ collapse to the same point $p\in Z$ (see case 4). In particular 
the method breaks down if the collision arises 
 for uniformly bounded values of $\Psi$ and $\nabla \Psi$, which means that the blowing up forces exerted between the shrinking  points may balance. Indeed in such a situation it is difficult to rule out \eqref{exp} with the couple $(\beta_1^n,\beta_2^n)=(1,0)$ since the left hand side remains bounded. 
In particular such \textit{crucial} collision does occur when  $\alpha_p=1$ for some $p\in Z$ and  two of the $N$ points $\xi_j^n$ are located symmetrically   
with respect to $p$ at distance $\rho_n\to 0^+,$ for instance, $$\xi_1^n=p+\rho_n e^{i\theta},\quad \xi_2^n=p+\rho_n e^{i(\theta+\pi)},\quad \theta \in \R.$$
Furthermore, another collision of this type  does occur when $\alpha_p=2$ for some $p\in Z$ and 
3 of the $N$ points $\xi_j^n$ are arranged at the vertices  of a regular triangle of radius $\rho_n\to 0^+$, for instance, 
$$\xi_1^n=p+\rho_n e^{i\theta},\quad \xi_2^n=p+\rho_ne^{i(\theta+\frac{2}{3}\pi)},\quad\xi_3^n=p+\rho_n e^{i(\theta+\frac{4}{3}\pi)},\quad \theta\in\R.$$
More generally, if $\alpha_p=n\in \{1,\ldots ,N-1\}$ for some $p\in Z$ the collisions associated to the  $(n+1)$-polygonal configuration centered at $p$ may cause a lack of compactness. This explains the assumption \eqref{akka2}.

\end{remark}
\appendix
\renewcommand{\theequation}{\Alph{section}.\arabic{equation}}

\section{Some properties of the Green's function}
Let $\Omega$ be a bounded domain with a ${\cal C}^2$-boundary. We denote by $G(x,y)$ the
Green's function of $-\Delta$ on $\Omega$ under Dirichlet boundary conditions, and by $H(x,y)$
its regular part, as in the introduction. So $H$ satisfies
\begin{equation*}
\left\{
\begin{aligned}
&\Delta_yH(x,y)=0 &\hbox{ }&y\in\Omega\\
&H(x,y)=-\frac{1}{2\pi}\log\frac{1}{|x-y|} &\hbox{ }&y\in\partial \Omega
\end{aligned}
\right..
\end{equation*}
We recall that $H$ is a smooth function in $\Omega\times \Omega$, $G$ and $H$ are
symmetric in $x$ and $y$, and $G>0$  in $\Omega\times\Omega$.
Moreover by the comparison principle we get
\beq\label{matis}-\frac{1}{2\pi}\log\frac{1}{|x-y|}\leq H(x,y)\leq C\qquad \forall x,y\in\Omega.\eeq 
The diagonal $H(x,x)$ is called the Robin's function of the domain $\Omega$ and satisfies
\begin{equation}\label{coercive}
H(x,x)\to -\infty\quad \hbox{ as }{\rm{d}}(x):={\rm{dist}}(x,\partial\Omega)
\to 0.
\end{equation}
\begin{remark}\label{remrem} Using the equation satisfied by $H$, we can obtain   $$ |\nabla_x H(x,y)|
=  O\bigg(\frac{1}{\di(x)}\bigg)
$$ uniformly for $x,\, y\in \Omega$.  To prove this, let us compute
\begin{equation}\label{bouesti2}
\bigg|\frac{\partial H}{\partial x_1}(x,y)
\bigg|
= \bigg|\frac{x_1-y_1}{2\pi| x- y|^2}\bigg|
  \leq \frac{1}{2\pi \di(x)}
\quad \forall x\in \Omega,\; \forall y\in\partial\Omega.
\end{equation} 
Now, for $x\in \Omega$ fixed, the function
$\frac{\partial H}{\partial x_1}(x,y)
$
is harmonic in $\Omega$ with respect to the variable $y$,
and verifies \eqref{bouesti2} on the boundary. The maximum principle  gives
$$
\bigg|\frac{\partial H}{\partial x_1}(x,y)
 \bigg|
 \leq \frac{1}{2\pi \di(x)}\quad \forall x,y\in\Omega.
$$
An analogous argument applies to $\frac{\partial H}{\partial x_2}(x,y)$.
\end{remark}

We  need the following  result concerning the behavior of the regular part $H(x,y)$ near the
boundary. To this aim we fix $\delta>0$ sufficiently small such that the projection onto
$\partial \Omega$ is well defined for any $x\in \R^2$ with  $\di(x)<\delta$. Then, setting   $\Omega_0:=\{x\in\Omega:{\rm d}(x)<\delta\}$,
we denote this projection by $p:\Omega_0 \to \partial\Omega$. It is of class ${\cal C}^1$ because
$\partial\Omega$ is of class ${\cal C}^2$. Moreover, for $x\in \Omega_0$, we write
$\bar x=2p(x)-x$ for the reflection of $x$ at $\partial \Omega$ and
$\nu_x=\frac{x-p(x)}{|x-p(x)|}$ for the inward unit normal at $p(x)$.
Let us observe that 
\beq\label{crais}| x-y|\leq |\bar x-x|+|\bar x-y|=2\di(\bar x)+|\bar x-y|\leq 3|\bar x-y|\quad \forall x\in\Omega_0,\,\forall y\in \Omega.\eeq

\begin{lemma}\label{robin}
Let $\Omega$ be a bounded domain with a ${\cal C}^2$-boundary. Then the following expansions
hold uniformly for $x\in \Omega_0$ and $y\in \Omega$:
$$
H(x,y)=-\frac{1}{2\pi}\log\frac{1}{|\bar x-y|}+O(1),
$$
and
$$
\frac{\partial H}{\partial \nu_x}(x,y)
=-\frac{1}{2\pi}\frac{\partial}{\partial\nu_x}\bigg(\log\frac{1}{|\bar x-y|}\bigg)
  +O(1)=\frac{1}{2\pi}\frac{(y-\bar x)\cdot \nu_x}{|\bar x- y|^2}+O(1).
$$

\end{lemma}

\begin{proof}
 For any $x\in \Omega_0$ we introduce a diffeomorphism which straightens
the boundary near $p(x)$. Let $T_x$ be a rotation and translation of coordinates which maps
$p(x)$ to $0$ and the unit inward normal $\nu_x$ to the vector ${\bf e}_2:=(0,1)$.
Then $T_x (x) = (0,{\rm{d}}(x))$, $T_x (\bar x) = (0,-{\rm{d}}(x))$,
and in some neighborhood of $0$ the boundary $\partial (T_x \Omega)$ can be represented by
$$
z_2=\rho_x(z_1);
$$
here $\rho_x$ is a $\cal C^2$-function satisfying $\rho_x(0)=0$ and $\rho'_x(0)=0$.
Therefore we have
$$
|z_2|\leq C |z_1|^2\quad \hbox{ on }\partial (T_x \Omega).
$$
First we prove the following estimate for the boundary points:
\begin{equation}\label{bouesti1}
\bigg|H(x,y)+\frac{1}{2\pi}\log\frac{1}{|\bar x-y|}\bigg|=\frac{1}{2\pi}\bigg|\log \frac{|x-y|}{|\bar  x-y|}\bigg|
 \leq C\quad \forall x\in\Omega_0,\;
\forall y\in \partial \Omega.
\end{equation}
In order to see this, we observe that for any  $x\in\Omega_0$ and $y\in\partial \Omega$, setting $z:=T_x(y)$, we have
\begin{equation}\label{dile2}
\max\{{\rm{d}}(x),|z_1|\} \le \min\{|x-y|,|\bar x-y|\},
\end{equation}
by which
\begin{equation}\label{lab}
\big||x-{y}|^{2}-|\bar x-y|^{2}\big|=4{\rm{d}}(x) |z_2|\leq C{\rm{d}}(x) |z_1|^2
 \leq C{\rm{d}}(x) \min\{|\bar x-{y}|^2, |x-y|^2\}.
\end{equation}
The above inequality implies
\begin{equation}\label{dill}
c\leq \frac{| x- y|}{|\bar x-y|}\leq C\quad \forall x\in \Omega_0, \;
\forall y\in\partial \Omega
\end{equation}
and \eqref{bouesti1} follows from \eqref{dill}. So, for any $x\in\Omega_0$,
the functions $H(x,y)+\frac{1}{2\pi}\log\frac{1}{|\bar x-y|}$ is 
 harmonic in $\Omega$ in the variable $y$, and verifies \eqref{bouesti1} on the boundary.
Then the maximum principle applies and gives
$$
\bigg|H(x,y)+\frac{1}{2\pi}\log\frac{1}{|\bar x-y|}\bigg|
\leq C
\quad \forall x\in \Omega_0, \;\forall y\in \Omega.
$$
The first part of the thesis follows.

We go on with the derivative estimate. We claim the following estimate on the boundary:
\begin{equation}\label{bouesti2l}
\bigg|\frac{\partial H}{\partial \nu_x}(x,y)
 -\frac{(y-\bar x)\cdot \nu_x}{2\pi|\bar x- y|^{2}}\bigg|
= \bigg|\frac{(x-y)\cdot \nu_x}{2\pi|x-y|^{2}}
  -\frac{(y-\bar x)\cdot \nu_x}{2\pi|\bar x-y|^{2}}\bigg|
\leq C
\quad \forall x\in \Omega_0,\; \forall y\in\partial\Omega.
\end{equation}
Indeed, using \eqref{dile2}-\eqref{lab}  we have
\begin{equation}\label{above}
\bigg|\frac{1}{|x-y|^{2}}-\frac{1}{|\bar x- y|^{2}}\bigg|
 \leq C\frac{{\rm{d}}(x)}{|\bar x- y|^2}\leq \frac{C}{|\bar x- y|}
 \quad \forall x\in \Omega_0,\; \forall y\in \partial \Omega.
\end{equation}
Moreover, for any  $x\in\Omega_0$ and  $y\in\partial \Omega$, setting  $z:=T_x(y)$, we get 
\begin{equation}\label{dile3}
\begin{aligned}
|(x-y)\cdot \nu_x-(y-\bar x)\cdot \nu_x|
 &=|({\rm{d}}(x) {\bf e}_2-z)\cdot {\bf e}_2-(z+{\rm{d}}(x) {\bf e}_2)\cdot {\bf e}_2|\\
 &=2|z_2|\leq C|z_1|^2\leq C| x-{y}|^2
\end{aligned}
\end{equation}
where for the last inequality we have used \eqref{dile2}.
Thus we obtain for $x\in \Omega_0$ and $y\in \partial \Omega$:
$$
\bigg|\frac{(x-y)\nu_x}{|x-y|^{2}}-\frac{( y-\bar x)\nu_x}{|\bar x- y|^{2}}\bigg|
 \leq |(y-\bar x) \nu_x|\bigg|\frac{1}{|x-y|^{2}}-\frac{1}{|\bar x- y|^{2}}\bigg|
       +\frac{|(x-y)\nu_x-(y-\bar x) \nu_x|}{| x- y|^{2}}
$$
and \eqref{bouesti2l} follows from  \eqref{above} and \eqref{dile3}.
Now, for $x\in \Omega_0$ fixed, the function
$\frac{\partial H}{\partial \nu_x}(x,y)-\frac{(y-\bar x)\cdot \nu_x}{2\pi|\bar x- y|^2}$
is harmonic in $\Omega$ with respect to the variable $y$,
and verifies \eqref{bouesti2l} on the boundary. The maximum principle applies and gives
$$
\bigg|\frac{\partial H}{\partial \nu_x}(x,y)
 -\frac{(y-\bar x)\cdot \nu_x}{2\pi|\bar x- y|^2}\bigg|
 \leq C\quad \forall x\in\Omega, \;\forall y\in \Omega_0.
$$
In order to conclude we observe that $\frac{\partial \bar x}{\partial \nu_x }=-\nu_x$, because
$\frac{\partial  p}{\partial \nu_x}(x)=0$ for any $x\in\Omega_0$, so that
$$
-\frac{\partial }{\partial \nu_x }\bigg(\log\frac{1}{|\bar x-y|}\bigg)
 =\frac{(y-\bar x)\cdot \nu_x}{|\bar x- y|^2}\quad \forall x\in\Omega_0,\;\forall y\in\Omega.
$$\end{proof}
Next two corollaries are devoted to estimate $G(x_n,y_n)$ on suitable sequences $x_n,\, y_n$.
\begin{corollary}\label{robin1} Let $\Omega$ be a bounded domain with a ${\cal C}^2$-boundary and let  $x_n, \, y_n\in \Omega$  be such that $|x_n-y_n|\geq c\di(x_n)$ for any $ n$. Then 
$$G(x_n, y_n)=O(1).$$
\end{corollary}
\begin{proof} The thesis is obvious if $|x_n-y_n|\geq c$, since, by \eqref{matis}, $H(x_n,y_n)=O(1)$ and, consequently, $$G(x_n,y_n)= \frac{1}{2\pi}\log\frac{1}{|x_n-y_n|}+H(x_n, y_n)=O(1).$$
Next assume $|x_n-y_n|=o(1)$, which implies $\di(x_n)=o(1)$ by hypothesis. Then we compute $$|\bar x_n-y_n|\leq |\bar x_n-x_n|+|x_n-y_n|=2\di(x_n)+|x_n-y_n|\leq C|x_n-y_n|.$$
So, combining this with  \eqref{crais},  $$c\leq \frac{|\bar x_n-y_n|}{|x_n-y_n|}\leq C$$by which, using Lemma \ref{robin}, 
$$G(x_n, y_n)=\frac{1}{2\pi}\log\frac{|\bar x_n-y_n|}{|x_n-y_n|}+O(1)=
O(1).$$

\end{proof}
\begin{corollary}\label{robin2}  Let $\Omega$ be a bounded domain with a ${\cal C}^2$-boundary and let   $x_n, \,y_n\in \Omega$  be such that $x_n,y_n\to \bar x\in\partial \Omega$. 
Then $$\frac{\partial H}{\partial \nu_{x_n}}(x_n,y_n)=\frac{1}{2\pi}\frac{(y_n-\bar x_n)\cdot \nu_{x_n}}{|\bar x_n- y_n|^2}+O(1)=\frac{1}{2\pi}\frac{\di(x_n)+\di(y_n)}{|\bar{x}_n-y_n|^2}+O(1).$$

 \end{corollary}
 \begin{proof}
 Since $\partial\Omega$ is of class    ${\cal C}^2$,  we get $$p(x_n)-p(y_n),\;\nu_{x_n}-\nu_{y_n}=O(|x_n-y_n|)$$ and 
 $$\frac{p(y_n)-p(x_n)}{|p(y_n)-p(x_n)|}=\pm \tau_{x_n}+O(|y_n-x_n|)$$ where $\tau_{x}$ denotes the tangent vector at $p(x)$; using these two facts we compute
$$\begin{aligned}(y_n-\bar{x}_n)\nu_{x_n}&=(y_n-2p(x_n)+x_n)\nu_{x_n}=
(y_n-p(x_n))\nu_{x_n}+\di(x_n)\\ &=(y_n-p(y_n))\nu_{x_n}+(p(y_n)-p(x_n))\nu_{x_n}+\di(x_n)\\ &=\di(y_n)\nu_{y_n}\nu_{x_n}+O(|y_n-x_n|^2)+\di(x_n)\\ &=(1+O(|x_n-y_n|))\di(y_n)+O(|y_n-x_n|^2)+\di(x_n).\end{aligned}$$ 
Then \eqref{crais}  together with the obvious inequality $|\bar{x}_n-y_n|\geq \di(y_n)$ yields 
$$\frac{y_n-\bar{x}_n}{|\bar{x}_n-y_n|^2}\nu_{x_n} =
\frac{\di(x_n)+\di(y_n)}{|\bar{x}_n-y_n|^2}+O(1)
.$$
 \end{proof}


\begin{thebibliography}{999}

\bibitem{bapa} S. Baraket, F. Pacard. {\em Construction of singular limits for a semilinear elliptic equation in dimension 2}, Calc. Var. Partial Differential Equations {\bf 6} (1998),  1--38.
\bibitem{bapi} D. Bartolucci, A. Pistoia. {\em Existence and qualitative properties of concentrating solutions for the $sinh$-Poisson equation}, IMA J. Appl. Math. {\bf 72} (2007),  706--729. 
\bibitem{bachelinta} D. Bartolucci, C.-C. Chen, C.-S. Lin, G. Tarantello. {\em Profile of blow-up solutions to mean field equations with singular data}, Comm. Partial Differential Equations {\bf 29} (2004),  1241--1265.
\bibitem{bata} D. Bartolucci, G. Tarantello. {\em  Liouville type equations with singular data and their application to periodic multivortices for the electroweak theory}, Comm. Math. Phys. {\bf 229} (2002), 3--47.

\bibitem{bapiwe} T. Bartsch, A. Pistoia, T. Weth. {\em $N$-vortex equilibria for ideal fluids in bounded planar domains and new nodal solutions of the $\sinh$-Poisson and the Lane-Emden-Fowler equations,} Commun. Math. Phys. {\bf 297} (2010), 653--686.
\bibitem{brebehe} F. Bethurel, H. Brezis, F. Helein. {\em Ginzburg-Landau vortices}, Birk\"auser, 1994.
\bibitem{breme} H. Brezis, F. Merle. {\em Uniform estimates and blow-up behavior for solutions of $-\Delta u = V (x)e^u$ in two dimensions},  Comm. Partial Differential Equations {\bf 16} (1991),  1223--1253.
\bibitem{chenlin} W. Chen, C. Li. 
{\em Classification of solutions of some nonlinear elliptic equations,}
Duke Math. J. {\bf 63} (1991), 615--623.

\bibitem{mio} T. D'Aprile. {\em Sign-changing blow-up solutions for H\'enon type elliptic equations with exponential nonlinearity}, preprint 2011.
\bibitem{delespomu} M. Del Pino, P. Esposito, M. Musso. {\em Two dimensional Euler flows with concentrated vorticities}, Trans. Amer. Math. Soc. {\bf 362} (2010), 6381--6395.
\bibitem{delfe} M. Del Pino, P. Felmer. {\em Semi-calssical states for nonlinear Schr\"odinger equations,} J. Funct. Anal. {\bf 149} (1997), 245--265.

\bibitem{delkomu} M. Del Pino, M. Kowalczyk, M. Musso. {\em Singular limits in Liouville-type equation,} Calc. Var. Partial Differential Equations {\bf 24} (2005), 47--81.
\bibitem{espo} P. Esposito. {\em Blow up solutions for a Liouville equation with singular data},  SIAM J. Math. Anal. {\bf  36} (2005), 1310--1345. 

\bibitem{espogropi} P. Esposito, M. Grossi, A. Pistoia. {\em On the existence of blowing-up solutions for a mean field equation}, Ann. Inst. H. Poincar\'e Anal. Non Lin\'eaire {\bf 22} (2005), 227--257.
\bibitem{espomupi1} P. Esposito, M. Musso, A. Pistoia. {\em On the existence and profile of nodal solutions for a two-dimensional elliptic problem with large exponent in nonlinearity} Proc. Lond. Math. Soc.  {\bf 94} (2007),  497--519.
\bibitem{espomupi2} P. Esposito, M. Musso, A. Pistoia. {\em Concentrating solutions for a planar elliptic problem involving nonlinearities with large exponent}, J. Differential Equations {\bf 227} (2006),  29--68.
\bibitem{espopiwei} P. Esposito, A. Pistoia, J. Wei. {\em Concentrating solutions for the H\'enon equation in $\R^2$}, J. Anal. Math. {\bf 100} (2006), 249--280.
\bibitem{espowei} P. Esposito, J. Wei. {\em Non-simple blow-up solutions for the Neumann two-dimensional $\sinh$-Gordon equation,} Calc. Var. Partial Differential Equations {\bf 34} (2009),  341--375.

\bibitem{lisha} Y.Y. Li, I. Shafrir. {\em Blow-up analysis for solutions of $-\Delta u = V e^u$ in dimension two}, Indiana Univ. Math. J. {\bf 43} (1994),  1255--1270.
\bibitem{mawe} L. Ma, J. Wei. {\em Convergence for a Liouville equation}, Comment. Math. Helv. {\bf 76} (2001),
506--514. 

\bibitem{nasu} K. Nagasaki, T. Suzuki. {\em Asymptotic analysis for two-dimensional elliptic eigenvalue problems
with exponentially dominated nonlinearities}, Asymptotic Anal. {\bf 3} (1990),  173--188.

\bibitem{su} T. Suzuki. {\em Two-dimensional Emden-Fowler equation with exponential nonlinearity, Nonlinear diffusion equations and their equilibrium states,} {\bf 3} (Gregynog, 1989), 493--512. Progr. Nonlinear Differential Equations Appl., 7, Birkh\"auser Boston, Boston, MA, 1992. 
\bibitem{ta1} G. Tarantello. {\em  Analytical aspects of Liouville-type equations with singular sources. Stationary partial differential equations. Vol. I,} 491--592, Handbook Differ. Equ., North-Holland, Amsterdam, 2004.
\bibitem{ta2} G. Tarantello. {\em A quantization property for blow up solutions of singular Liouville-type equations}, J. Funct. Anal. {\bf 219} (2005), 368--399.

\bibitem{yang} Y. Yang. {\em Solitons in field theory and nonlinear analysis}, Springer-Verlag, New York, 2001.
\bibitem{weiyezhou} J. Wei, D. Ye, F. Zhou. {\em Bubbling solutions for an anisotropic Emden-Fowler equation},  Calc. Var. Partial Differential Equations {\bf 28} (2007), 217--247. 

\bibitem{we} V.H. Weston. {\em On the asymptotic solution of a partial differential equation with an exponential nonlinearity,} SIAM J. Math. Anal. {\bf 9} (1978),  1030--1053. 

\end{thebibliography}
\end{document}